\documentclass[12pt,leqno]{amsart}
\usepackage[latin1]{inputenc}
\usepackage{amsmath}
\usepackage{amssymb, amsfonts,amscd,verbatim,hyperref,cleveref,graphics}

\usepackage{xspace}
\usepackage{tikz}
\usetikzlibrary{matrix,calc,decorations.pathmorphing,shapes,arrows}
\usepackage{tikz-cd} 
\tikzset{
  symbol/.style={
    draw=none,
    every to/.append style={
      edge node={node [sloped, allow upside down, auto=false]{$#1$}}}
  }
}

\newtheorem{theorem}{Theorem}[section]
\newtheorem{proposition}[theorem]{Proposition}
\newtheorem{corollary}[theorem]{Corollary}
\newtheorem{lemma}[theorem]{Lemma}
\numberwithin{equation}{section}

\theoremstyle{definition}
\newtheorem{remark}[theorem]{Remark}
\newtheorem{example}[theorem]{Example}
\newtheorem{question}[theorem]{Question}

\newcommand{\s}{\sum_{k=1}^\infty}

\newcommand{\norm}[1]{\lVert#1\rVert}

\newcommand{\Bignorm}[1]{\Bigl\lVert#1\Bigr\rVert}

\newcommand{\term}[1]{{\textit{#1}}}   
\newcommand{\marg}[1]{\marginpar{\tiny #1}}     

\newcommand{\timur}[1]{{\textcolor{blue}{{\bf T.O:} #1}}}

\DeclareSymbolFont{bbold}{U}{bbold}{m}{n}
\DeclareSymbolFontAlphabet{\mathbbold}{bbold}
\def\one{\mathbbold{1}}

\newcommand{\zs}

\definecolor{dred}{RGB}{95,2,31}

\usepackage{mathtools}

\usepackage{tabularx,hhline}
\usepackage{graphicx}
\usepackage{enumitem}
\newlist{primenumerate}{enumerate}{1}
\setlist[primenumerate,1]{label={\arabic*$'$}}
\def\vp{\varepsilon}

\def\N{{\mathbb N}}
\def\R{{\mathbb R}}

\newcommand{\spn}{\mathrm{span}}

\begin{document}

\title[Maximal inequalities, frames and greedy algorithms]
{Maximal inequalities, frames and greedy algorithms}

\author{Pablo~M.\ Bern\'a}
\address{Departamento de Matem\'aticas, CUNEF Universidad\\ Madrid 28040, Spain
} 
\email{pablo.berna@cunef.edu}

\author{ Daniel\ Freeman}
\address{Department of Mathematics and Statistics\\
St Louis University\\
St Louis MO 63103,  USA
} \email{daniel.freeman@slu.edu}

\author{Timur\  Oikhberg}
\address{Deptartment of Mathematics, University of Illinois, Urbana IL 61801, USA}
\email{oikhberg@illinois.edu}

\author{Mitchell~A.\ Taylor}
\address{Department of Mathematics\\
ETH Z\"urich, Ramistrasse 101 8092 Z\"urich, Switzerland
} \email{mitchell.taylor@math.ethz.ch}


\date{\today}

\begin{abstract}
The aim of this article is to use Banach lattice techniques to study coordinate systems in function spaces. We begin by proving that the greedy algorithm of a basis is order convergent if and only if a certain maximal inequality is satisfied. We then show that  absolute frames need not admit a reconstruction algorithm with respect to the usual order convergence, but do allow for reconstruction with respect to the order convergence inherited from the double dual. After this, we investigate the extent to which such coordinate systems affect the geometry of the underlying function space. Most notably, we prove that a Banach lattice $X$ is lattice isomorphic to a closed sublattice of a $C(K)$-space if and only if every unconditional sequence in $X$ is absolute.

\end{abstract}

\maketitle
\tableofcontents
\section{Introduction}\label{intro sec}
The study of coordinate systems in Banach spaces (Schauder bases, Markushevich bases, decompositions, frames, greedy algorithms) is a classical subject. However, in applications, the desirable coordinate systems often have additional structure, which may not even make sense in a generic Banach or Hilbert space. Common examples include wavelets, which make use of dilation and translation, and almost everywhere convergence systems.  Evidently, to even define these concepts, one must be working in  spaces  with suitable symmetries, or possessing additional notions of convergence. 
\medskip

Here, we will consider the interplay between coordinate systems and lattice structure. When considering a basic sequence $(x_k)$ in a Banach lattice, there are many different ways to incorporate the partial ordering. The most common requirement is to place restrictions on the set $\{x_k\}$. For example, one may require that this set be contained in the positive cone, an order interval, or be pairwise disjoint. The issue with this approach is that bases in function spaces cannot be disjoint, and those that appear in practice are rarely contained in the positive cone.  
Indeed, it is non-trivial \cite{FPT,JS} to even construct a positive basis in $L_1(\mathbb{R})$ and $L_2(\mathbb{R})$, and it is not known whether positive bases  exist in $L_p(\mathbb{R})$ for $p\neq 1,2.$ Moreover, if $(x_k)$ is a normalized basis for $L_1[0,1],$ then $\{x_k\}$ cannot be almost order bounded (equivalently, equiintegrable) \cite[p.~74-75]{DJT}. Hence, requiring that  $\{x_k\}$ lie in an order interval is also very restrictive.
\medskip

In contrast to the above, our approach in this article will be to  utilize  the lattice structure to define \emph{maximal functions} and \emph{order convergence}. For the moment, let us simply note that in any vector lattice $X$ one may define notions of order convergence $f_k\xrightarrow{o}f$ and uniform convergence $f_k\xrightarrow{u}f$. Moreover, when $X$ is a space of measurable functions, we have that $f_k\xrightarrow{o}f$ if and only if $f_k\xrightarrow{a.e.}f$ and there is a $g\in X$ with $|f_k|\leq g$ for all $k$. For this reason, one may view order convergence as a generalization of dominated almost everywhere convergence to vector lattices. 
\medskip

 The starting point of our paper is a result of  \cite{TT19} which states that for a basic sequence $(x_k)$ in a Banach lattice $X$, establishing order convergence of the basis expansions is equivalent to establishing boundedness of the associated maximal function. More precisely, letting $P_n(\sum_{k=1}^\infty a_kx_k)=\sum_{k=1}^n a_kx_k$ denote the $n$-th canonical basis projection, we have the following theorem.
\begin{theorem}[\cite{TT19} Theorem 3.1]\label{bidecomposition}
Let $(x_k)$ be a basic sequence in a Banach lattice $X$. Denote by $[x_k]$ its closed linear span and let $P_n:[x_k]\to [x_k]$  denote the $n$-th canonical basis projection. The following are equivalent.
\begin{enumerate}
\item For all $x\in [x_k]$, $P_nx\xrightarrow{u}x$;
\item For all $x\in [x_k]$, $P_nx\xrightarrow{o}x$;
\item For all $x\in [x_k]$,  $|P_nx|\leq u$ for some $u\in X$ and  all $n$;
\item For all $x\in [x_k]$, $(\bigvee_{n=1}^m\left|P_nx\right|)_m$ is norm bounded;
\item There exists $M\geq 1$ such that for all $m\in \mathbb{N}$ and scalars $a_1,\dots, a_m$  one has 
\begin{equation}\label{bidecomp inequality}
\Bigg\|\bigvee_{n=1}^m\left|\sum_{k=1}^na_kx_k\right|\Bigg\|\leq M\Bigg\|\sum_{k=1}^m a_kx_k\Bigg\|.
\end{equation}
\end{enumerate}
\end{theorem}
The motivation for \Cref{bidecomposition} stems from a well-known principle  in harmonic analysis \cite{MR0189110,MR0209867,MR0125392} which states that for a family of linear operators $(T_k)_{k\in \mathbb{N}}$ mapping $L_p(\Omega)$ to $L_0(\Omega)$, establishing almost everywhere convergence $T_kf\xrightarrow{a.e.}f$  for all $f\in L_p(\Omega)$ is equivalent to establishing an inequality of weak type $(p,p)$ for the associated maximal operator $T^*f(x):=\sup_k|T_kf(x)|$.  From this perspective, \Cref{bidecomposition}   states that -- in the full generality of Banach lattices -- establishing \emph{dominated} almost everywhere convergence of the basis expansions  is equivalent to establishing strong boundedness of the associated maximal operator. 
\medskip

A second motivation for \Cref{bidecomposition} is that most of the important bases in martingale theory, harmonic analysis, probability, stochastic processes and
orthogonal series  do possess order convergent sums, although proving this is often a major result.  Common examples include martingale difference sequences in $L_p(P)$ with $p>1$ and $P$ a probability measure (this is essentially Doob's inequality), the Walsh basis (see \cite{MR0241885}), and unconditional blocks of the Haar in $L_1[0,1]$ (this is essentially Burkholder-Davis-Gundy combined with Khintchine). The proof of the
Carleson-Hunt theorem in \cite{Hunt} establishes \eqref{bidecomp inequality}  for the
trigonometric basis. It is also worth mentioning the work of Bourgain
\cite{B89} who used probabilistic techniques to make progress on the
Kolmogorov and Garsia conjectures. In our language, these conjectures
essentially ask whether every orthonormal basis of $L_2(P)$ admits a 
rearrangement satisfying \Cref{bidecomposition}.  In \cite{B89}, Bourgain was able to construct a rearrangement satisfying the inequality \eqref{bidecomp inequality} with $M\sim \log(\log(m))$, which is the optimal dependence that one can achieve for random rearrangements. More generally, maximal inequalities such as \eqref{bidecomp inequality}  have a long history in analysis, and can be used to prove the Birkhoff ergodic theorem \cite{Stein}, a.e.~convergence of the Schr\"odinger evolution back to the initial data \cite{DZ19}, and even appear in the non-commutative setting \cite{HRW,  JX07, P93, S20}.  
\medskip


The  objective of this paper is twofold. First, we wish  to continue the study initiated in \cite{ARTT,GKP, TT19} of basic sequences satisfying \Cref{bidecomposition}. Secondly, we want to extend the theory to frames and greedy algorithms. Informally, frames can be thought of as redundant bases. The extra redundancy present in a frame allows for a more robust reconstruction  with respect to the norm convergence, which is crucial in many applications. However, it also makes reconstruction with respect to  order and uniform convergence  much more subtle. Indeed, as we will see,  \Cref{bidecomposition} completely fails for frames, yet certain natural strengthenings of \eqref{bidecomp inequality} do guarantee a reconstruction algorithm with respect to the order convergence inherited from the double dual. 
\medskip

In contrast,  an analogue of \Cref{bidecomposition} holds in complete generality  for the  greedy algorithm. Recall that the greedy algorithm is a nonlinear approximation scheme which approximates vectors using the coefficients of a basis with largest magnitude. Such algorithms appear  in applications such as signal processing, where one needs to approximate a signal $x=\sum_k a_k x_k$ in an infinite or high dimensional Banach space using only a relatively small number of coordinates $\sum_{k=1}^n b_{m_k} x_{m_k}$.  For orthornormal bases in Hilbert spaces, the most effective  way to approximate $x$ in norm using $n$ coordinates is to choose the $n$ coordinates of the basis expansion with the largest magnitude.  That is,  we set $\mathcal{G}_n(x) := \sum_{k=1}^n a_{m_k} x_{m_k}$ where $|a_{m_k}|\geq |a_i|$ for all $1\leq k\leq n$ and $i\not\in\{m_k\}_{1\leq k\leq n}$. Although it is easy to see that such a scheme provides an optimal  approximation  when $(x_k)$ is orthonormal, it also turns out to be  highly efficient for a large class of bases in Banach spaces. See, for example, \cite{MR2848161} for a survey on the applications and  effectiveness of the greedy algorithm. 
\medskip

Given the interest in determining whether the greedy algorithm provides an effective approximation in norm, it is also of interest to know how well it approximates a vector in order. As it turns out, by  considering the convergence of the greedy sums in \Cref{bidecomposition} rather than the convergence of the partial sums over initial segments of the basis, we are able to obtain the following nonlinear variant of the equivalence between establishing order convergence and maximal function estimates.
\begin{theorem}\label{u-greedy theorem intro}
Let $(x_k)$ be a semi-normalized basic sequence in a Banach lattice $X$ and let $[x_k]$ denote its closed linear span.  The following are equivalent.
\begin{enumerate}
\item For all $x\in [x_k]$,  $\mathcal{G}_n(x)\xrightarrow{u}x$;
\item For all $x\in [x_k]$, $\mathcal{G}_n(x)\xrightarrow{o}x$;
\item For all $x\in [x_k]$, $|\mathcal{G}_n(x)|\leq u$ for some $u\in X$ and all $n$;
\item For all $x\in [x_k]$, $(\bigvee_{n=1}^m|\mathcal{G}_n(x)|)_{m}$ is norm bounded;
\item There exists $C\geq 1$ such that for all $x\in [x_k]$ and $m\in \mathbb{N}$,
$$\bigg\|\bigvee_{n=1}^m\left|\mathcal{G}_n(x)\right|\bigg\|\leq C\|x\|.$$
\end{enumerate}
\end{theorem}
One of the main goals of this article is to study the implications of \Cref{u-greedy theorem intro}. For now, we note that \Cref{u-greedy theorem intro} is a theorem about basic sequences which can have implications outside the context of Banach lattices.  That is, it is possible to study different properties of a basic sequence $(x_k)$ by embedding $(x_k)$ into different Banach lattices and then applying \Cref{u-greedy theorem intro}.  In particular, since $[x_k]$ is separable, we may consider $[x_k]$ as a subspace of $C[0,1]$. In $C[0,1]$,  uniform convergence agrees with norm convergence and $\| |x|\vee |y|\|=\|x\|\vee\|y\|$ for all $x,y\in C[0,1]$. Thus, choosing $X=C[0,1]$ in \Cref{u-greedy theorem intro} we recover the following key theorem about quasi-greedy bases from \cite{W00} which, a priori, makes no mention of Banach lattices. 
\begin{theorem}[\cite{W00} Theorem 1; \cite{AABW19} Theorem 4.1]\label{W00}
Let $E$ be a Banach space with a semi-normalized Schauder basis $(x_k)$. The following are equivalent.
\begin{enumerate}
\item For all $x\in E$,  $\mathcal{G}_n(x)\xrightarrow{\|\cdot\|}x$;
\item There exists $C\geq 1$ such that for all $x\in E$ and $n\in \mathbb{N}$,
$\|\mathcal{G}_n(x)\|\leq C\|x\|.$
\end{enumerate}
\end{theorem}
Although \Cref{u-greedy theorem intro}  reduces to \Cref{W00} when $X=C[0,1]$, when $X=L_p(\Omega)$ it is  often highly non-trivial to verify whether a  basis satisfies \Cref{u-greedy theorem intro}. Several examples and structural properties of such bases will be given throughout the paper, but for now let us note that Tao \cite{Tao96} proved that all compactly supported wavelet bases  satisfy  \Cref{u-greedy theorem intro} whereas  K\"orner \cite{K96, K06} proved that the Fourier and Walsh bases do not. In other words,  even though the Fourier and Walsh bases satisfy \Cref{bidecomposition} and are optimally approximated  by the greedy algorithm in norm, there are $f\in L_2([0,1])$  for which $(\mathcal{G}_nf)_{n=1}^\infty$ fails to converge to $f$ almost everywhere. In fact, K\"orner could construct $f\in L_2([0,1])$ such that the greedy algorithm for $f$ diverges at almost every point.
\subsection{Outline of the paper}
We now briefly summarize  the paper.  In \Cref{Section2} we recall the necessary vector lattice and Banach space terminology. Then, in \Cref{Section3} we combine the ideas of \cite{T,TT19} with those from greedy theory to generalize several results on quasi-greedy bases to the lattice setting. More specifically, in \Cref{Sub1} we prove  \Cref{u-greedy theorem intro}, in \Cref{frame subsec} we discuss the extent to which \Cref{bidecomposition} holds for frames (and, as a corollary of our methods, solve a problem from \cite{hart2023overcomplete}), and in \Cref{Sub3} we  verify that most of the fundamental results on quasi-greedy bases extend to uniformly quasi-greedy bases. Since the results in Sections~\ref{Sub1} and \ref{Sub3} must recover the classical results on quasi-greedy bases when $X=C[0,1]$,  there is some inevitable overlap between the new and classical proofs -- our exposition in these two subsections will loosely follow  \cite[Chapter 10]{AK16}. 
\medskip


The heart of the paper is \Cref{s:destroy} where we aim to develop a qualitatively new perspective on coordinate systems in Banach lattices by proving  results which have no direct analogues in the classical theories of Schauder or quasi-greedy bases. We begin in \Cref{embed bad} with the goal of characterizing when a ``good" basis of a Banach space $E$ can be embedded into a Banach lattice $X$ so that it inherits ``bad" lattice properties. As already mentioned, by choosing $X=C[0,1]$ one can always embed a Schauder or quasi-greedy basis into a lattice so that Theorems~\ref{bidecomposition} and \ref{u-greedy theorem intro} hold. In \Cref{ss:destroy_Lind} we construct a Banach lattice $X$ containing a copy of the Lindenstrauss basis which  simultaneously fails  Theorems~\ref{bidecomposition} and \ref{u-greedy theorem intro}. Then, in \Cref{ss:destroy_lp} we construct, for each $p>1$, a basis $(u_i)$ of $\ell_p$ which is equivalent to the canonical basis of $\ell_p$ yet simultaneously fails   Theorems~\ref{bidecomposition} and \ref{u-greedy theorem intro}. Finally, we show in \Cref{ss:uncondiitonal bases} that  a normalized basis  can  be embedded  into a Banach lattice in such a way that it fails to be absolute (i.e., so that the  maximal inequality for unconditional bases fails) if and only if it is not equivalent to the canonical  $\ell_1$ basis.
\medskip

In \Cref{Section5} we move from embeddings to blockings. This is a topic where  bibases and greedy bases behave very differently. Indeed, as we will demonstrate in \Cref{section 5.1},  it is quite easy to block a basis so that it satisfies \Cref{bidecomposition} or its unconditional variant. On the other hand, in \Cref{lose} we show that the canonical basis of $\ell_p$  is the only subsymmetric basis for which every blocking is greedy. 
\medskip

In \Cref{ss:bad_basic} we characterize (up to a lattice isomorphism) AM-spaces as the only Banach lattices for which every unconditional sequence is absolute.  We also show that if every basic sequence in a Banach lattice $X$ satisfies \Cref{bidecomposition} then $X$ must be  $p$-convex for all finite $p$.
\medskip

In \Cref{complemented} we prove that complemented absolute sequences behave very much like disjoint sequences and use this to give a Banach lattice proof of the well-known fact that the only complemented  subspace of $C[0,1]$ with an unconditional basis is $c_0$. In \Cref{Ambient} we prove that every unconditional basis of $L_p$ can be rearranged to fail \Cref{bidecomposition} and that  $\sigma$-order complete Banach lattices which embed into the span of an absolute FDD  must be purely atomic. \medskip

Throughout, we make use of standard facts and notation -- see e.g.~\cite{AK16, LT1, LT2} for Banach spaces in general, \cite{MN} for Banach lattices specifically. The field of scalars is $\R$, unless otherwise specified.
\medskip

\textbf{Acknowledgments.} The first author was supported by the grant PID2022-142202NB-I00 / AEI / 10.13039/501100011033 (Agencia Estatal de Investigaci\'on, Spain).



\section{Background}\label{Section2}
In this section, we provide the necessary vector lattice and Banach space background -- background on frames and greedy algorithms will appear throughout the paper. 
\medskip

If $(y_\alpha)$ is a net in a vector lattice $X$ then  $y_{\alpha}\downarrow 0$ means that $y_\alpha$ is decreasing and $\inf\{y_\alpha\}=0$.  A vector lattice $X$ is called {\em Archimedean} if $n^{-1}e\downarrow 0$ for every positive vector $e\in X$.  We will be working with the
following three classical notions of sequential convergence of a sequence $(x_n)_{n=1}^\infty$ in an Archimedean vector lattice $X$.
\begin{itemize}
\item We say that $(x_n)_{n=1}^\infty$ {\em uniformly converges} to $x$ and write $x_n\xrightarrow{u}x$ if there exists $e\in X_+$ and a sequence $\epsilon_m\downarrow 0$ in $\mathbb{R}$ satisfying: $$\forall m\ \exists n_m \ \forall n\geq n_m,\ |x_n-x|\leq\epsilon_m e.$$
\item We say that $(x_n)_{n=1}^\infty$ {\em order converges} to $x$ (with respect to a sequence)  and write $x_n\xrightarrow{o_1}x$ if there exists a sequence $y_m\downarrow 0$ in $X$ satisfying: $$\forall m\ \exists n_m \ \forall n\geq n_m, \ |x_n-x|\leq y_m.$$
\item We say that $(x_n)_{n=1}^\infty$ {\em order converges} to $x$ (with respect to a net)  and write $x_n\xrightarrow{o}x$ if there exists a net $y_{\beta}\downarrow 0$ in $X$ satisfying: $$\forall \beta\ \exists n_{\beta} \ \forall n \geq n_{\beta}, \ |x_n-x|\leq y_{\beta}.$$
\end{itemize}
 Evidently, $x_n\xrightarrow{u}x\Rightarrow x_n\xrightarrow{o_1}x \Rightarrow x_n\xrightarrow{o}x$; however, neither converse implication holds in general. In a Banach lattice one also has norm convergence, and although  $u$-convergence implies norm convergence, $o_1$-convergence need not. It is also not true that norm convergent sequences are $o$-convergent, but they do at least have $u$-convergent subsequences. 
 \medskip

Recall that a \emph{(Schauder) basis} of a Banach space $X$ is a sequence $(x_k)$ in $X$ such that  for
every vector $x$ in $X$ there is a unique sequence of scalars $(a_k)$ satisfying
$x=\s a_k x_k$. Here, of course, the series converges in norm.
For each~$n$,
we define the $n$-th basis projection $P_n\colon X\to X$ via
\begin{math}
  P_n\bigl(\s a_k x_k)=\sum_{k=1}^na_k x_k,
\end{math}
so that $P_nx\xrightarrow{\|\cdot\|}x.$
It is known that the projections $P_n$ are uniformly bounded; the number
$K:=\sup_n\norm{P_n}$ is called the basis constant of~$(x_k)$. A
sequence $(x_k)$ in $X$ is called a \emph{(Schauder) basic sequence}
if it is a Schauder basis of its closed linear span~$[x_k]$. In this case, the
$P_n$'s are defined on~$[x_k]$. It is a standard fact
that a sequence $(x_k)$ of non-zero vectors  is Schauder basic iff there
exists a constant $K\ge 1$ such that
\begin{equation}\label{reg basis in}
  \bigvee_{n=1}^m\Bignorm{\sum_{k=1}^na_kx_k}
  \le K\Bignorm{\sum_{k=1}^ma_kx_k}\hspace{1cm}\textrm{ for all $m\in\N$ and all scalars $a_1,\dots,a_m$.}
\end{equation}
Moreover, the least
value of the constant $K$ is the basis constant of~$(x_k)$.  A  sequence $(x_k)$ is called \emph{unconditional} if every permutation of $(x_k)$ is basic. As is well-known, unconditionality of a basic sequence is characterized by the existence of a constant $C\geq 1$ such that
\begin{equation}\label{reg basis unc}
  \bigvee_{\epsilon_k=\pm 1}\Bignorm{\sum_{k=1}^m \epsilon_ka_kx_k}
  \leq C\Bignorm{\sum_{k=1}^ma_kx_k}\hspace{1cm}\textrm{ for all $m\in\N$ and all scalars $a_1,\dots,a_m$.}
\end{equation}

Suppose now that $X$ is a Banach lattice.  By interchanging the order of the supremum and the norm, one obtains the ``maximal inequality" variants of \eqref{reg basis in} and \eqref{reg basis unc}.  More formally, 
a sequence  $(x_k)$ of non-zero vectors in a Banach lattice $X$ is called \emph{bibasic} if there exists a constant $M\geq 1$ such that 
\begin{equation}\label{reg basis in2}
 \Bignorm{ \bigvee_{n=1}^m\bigg|\sum_{k=1}^na_kx_k\bigg|}\le M\Bignorm{\sum_{k=1}^ma_kx_k}\hspace{1cm}\textrm{ for all $m\in\N$ and all scalars $a_1,\dots,a_m$}
\end{equation}
and \emph{absolute} if there exists a constant $A\geq 1$ such that 
 \begin{equation}\label{absolute basis unc }
\Bignorm{\sum_{k=1}^m |a_kx_k|}
  \leq A\Bignorm{\sum_{k=1}^ma_kx_k}\hspace{1cm}\textrm{ for all $m\in\N$ and all scalars $a_1,\dots,a_m$.}
\end{equation}
Note that in obtaining the inequality \eqref{absolute basis unc } from \eqref{reg basis unc} we used the standard fact that 
\begin{equation}\label{standard sup}
    \bigvee_{\epsilon_k=\pm 1}\left|\sum_{k=1}^m \epsilon_ka_kx_k\right|=\sum_{k=1}^m |a_kx_k|,
\end{equation}
which holds in any Archimedean vector lattice. Clearly, every bibasic sequence is basic.  By \Cref{bidecomposition}, a basic sequence is bibasic  if and only if, for each $x\in [x_k]$, $P_nx\xrightarrow{o}x$. 
\medskip

  When $X$ is a $C(K)$-space it is easy to see  that every basic sequence is  bibasic and   every unconditional sequence is absolute. One of the main objectives of this paper  is to prove a converse statement; namely, that a Banach lattice $X$ is lattice isomorphic to a closed sublattice of a $C(K)$-space if and only if every unconditional sequence is absolute.

\section{Reconstruction algorithms for coordinate systems in Banach lattices}\label{Section3}
In this section, we begin our goal of infusing lattice technique into the study of frames and greedy bases. 
\subsection{A greedy version of the bibasis theorem}\label{Sub1}
We begin by proving \Cref{u-greedy theorem intro}. To fix the notation, $X$ will be a Banach lattice, $E$ a closed subspace of $X$, and $\mathcal{B}=(e_n)$ a semi-normalized basis of $E$. The biorthogonal functionals of $(e_n)$ will be denoted by $e_n^*\in E^*$. 
\medskip

We first recall some of the essentials on greedy algorithms. Fix $x\in E$. A \emph{greedy ordering} of $x$ is an injective map $\pi : \mathbb{N}\to \mathbb{N}$ such that $\{n: e_n^*(x)\neq 0\}\subseteq \pi(\mathbb{N})$ and  $(|e^*_{\pi(n)}(x)|)$ is non-increasing. The \emph{$m$-th greedy sum of $x$} associated to $\pi$ is given by $G_{\pi,m}(x):=\sum_{n=1}^m e^*_{\pi(n)}(x)e_{\pi(n)},$
and the sequence $(G_{\pi,m}(x))_{m=1}^\infty$ is called a \emph{greedy approximation of $x$}.  A  \emph{strictly greedy sum of $x$ of order $m$} is an element  $G_{\pi,m}(x)$ for which  there is no ambiguity in defining the $m$-th greedy sum of $x$, i.e., $G_{\pi,m}(x)=G_{\sigma,m}(x)$ for all  greedy orderings $\sigma$.
It is well-known that the choice of greedy ordering is not particularly important, and hence we will usually use the ordering induced by the basis.  More specifically, we define  $\rho: \mathbb{N}\to \mathbb{N}$ so that $\{n: e_n^*(x)\neq 0\}\subseteq \rho(\mathbb{N})$ and if $j<k$ then either $|e_{\rho(j)}^*(x)|>|e_{\rho(k)}^*(x)|$ or $|e^*_{\rho(j)}(x)|=|e_{\rho(k)}^*(x)|$ and $\rho(j)<\rho(k)$. We define the \emph{$m$-th  natural greedy sum of $x$} as $\mathcal{G}_m(x):=\sum_{n=1}^me_{\rho(n)}^*(x)e_{\rho(n)}$ and call the basis $(e_n)$  \emph{quasi-greedy} if, for all $x\in E$, $\mathcal{G}_m(x)\xrightarrow{\|\cdot\|}x$. By \Cref{W00}, $(e_n)$ is quasi-greedy iff there exists a constant  $C\geq 1$ such that for all $x\in E$ and $m\in \mathbb{N}$,
$\|\mathcal{G}_m(x)\|\leq C\|x\|.$
\medskip




All of the above definitions are valid in Banach, and even quasi-Banach spaces. Since we require $X$ to be a lattice, for each $x\in E$ and greedy ordering $\pi$ we may define $$G_{\pi,m}^\vee(x):=\bigvee_{n=1}^m\left|G_{\pi,n}(x)\right|=\bigvee_{n=1}^m\left|\sum_{k=1}^n  e_{\pi(k)}^*(x)e_{\pi(k)}\right|.$$ In particular, we may consider the map $x\mapsto\mathcal{G}^\vee_m(x):=\bigvee_{n=1}^m|\mathcal{G}_n(x)|$, and because it is absolutely homogeneous  we may define
$\|\mathcal{G}^\vee_m\|:=\sup_{x\in S_E}\|\mathcal{G}^\vee_m(x)\|.$ We say that $G_{\pi,m}(x)$ is a \emph{lattice strictly greedy sum of $x$} of order $m$ if $|e^*_{\pi(i)}(x)|\neq |e^*_{\pi(j)}(x)|$ for each distinct $i,j\in \{1,\dots,m\}$ for which $e^*_{\pi(i)}(x)\neq 0$.
\medskip

To begin, we extend the fact  that the choice of greedy ordering is  not particularly important to the lattice setting. In the proof, we need to take some care since $G^\vee_{\pi,m}(x)$ may depend on $\pi$, even when $G_{\pi,m}(x)$ is strictly greedy.

\begin{lemma}\label{equivalence of constants.}
Let $(e_n)$ be a semi-normalized basic sequence in a Banach lattice $X$ and let $C>0$ be a constant. The following are equivalent.
\begin{enumerate}
\item For all $x\in E$, all $m\in \N$ and all greedy orderings $\pi$, $\|G_{\pi,m}^\vee(x)\|\leq C\|x\|.$
\item For all $x\in E$ and all $m\in \N$, $\|\mathcal{G}^\vee_m(x)\|\leq C\|x\|$.
\item For all $x\in E$ and all $m\in \N$ there exists a greedy ordering $\pi$ such that $\|G_{\pi,m}^\vee(x)\|\leq C\|x\|.$
\item  We have $\|{\mathcal{G}^\vee_{|\text{supp} (x)|}}(x)\|\leq C\|x\|$ whenever $x\in E$ has finite support and $\mathcal{G}_{|\text{supp}(x)|}(x)$ is lattice strictly~greedy.
\item We have $\|{\mathcal{G}^\vee_{|\text{supp} (x)|}}(x)\|\leq C\|x\|$ for all $x\in E$ of finite support.

\end{enumerate}
Furthermore, the least such $C$ so that the above estimates hold is  $C=\sup_m\|\mathcal{G}^\vee_m\|$.
\end{lemma}
\begin{proof}
(i)$\Rightarrow$(ii)$\Rightarrow$(iii)$\Rightarrow$(iv) is clear. 
\medskip

(iv)$\Rightarrow$(v): Let $x\in E$ have finite support, say, $r$. Write $x=\mathcal{G}_r(x)=\sum_{n=1}^r e_{\rho(n)}^*(x)e_{\rho(n)}$. Fix $\varepsilon>0$ and find $\delta_1, \dots, \delta_r$ with $|e_{\rho(1)}^*(x)+\delta_1|>\cdots >|e_{\rho(r)}^*(x)+\delta_r|>\max_{i\not\in \{\rho(1),\dots,\rho(r)\}}|e_i^*(x)|$ and $\|\delta_n e_{\rho(n)}\|<\frac{\varepsilon}{Cm}$ for all $n=1,\dots,r$. We  consider the element $y:=x+\sum_{n=1}^r \delta_n e_{\rho(n)}$. It is easy to see that $y=\mathcal{G}_r(y)$ is a lattice strictly greedy sum,  so  (iv) implies that
$$\bigg\|\bigvee_{i=1}^r \left|\sum_{n=1}^i (e_{\rho(n)}^*(x)+\delta_n)e_{\rho(n)} \right|\bigg\|\leq C\|y\|\leq C\|x\|+\varepsilon.$$
Hence,
\begin{eqnarray*}
\bigg\|\bigvee_{i=1}^r \left|\sum_{n=1}^i e_{\rho(n)}^*(x)e_{\rho(n)} \right|\bigg\|&\leq& \bigg\|\bigvee_{i=1}^r \left|\sum_{n=1}^i (e_{\rho(n)}^*(x)+\delta_n)e_{\rho(n)} \right|\bigg\|\\
&+&\bigg\||\delta_1e_{\rho(1)}|+\cdots +|\delta_re_{\rho(r)}|\bigg\|\\
&\leq& C\|x\|+2\varepsilon,
\end{eqnarray*}
which yields that $\|\mathcal{G}_m^\vee(x)\|\leq C\|x\|$.
\medskip



(v)$\Rightarrow$(i): By a similar argument as (iv)$\Rightarrow$(v), the condition (v) implies that 
$\|G_{\pi,|\text{supp}(x)|}^\vee(x)\|\leq C\|x\|$ for all $x\in E$ of finite support and all greedy orderings $\pi$ of $x$. We now fix $m\in\N$ and let $\pi$ be a greedy ordering of $x$. Note that for any $n\in\N$ satisfying $n\geq \max(\pi(1),...,\pi(m))$ we have  
\begin{equation}\label{ine lim}
    \|G_{\pi,m}^\vee(x)\|=\|G_{\pi,m}^\vee(P_n x)\|\leq \|G_{\pi,n}^\vee(P_n x)\|\leq C\|P_n x\|.
\end{equation}
On the other hand, since $(e_n)$ is a basic sequence, we have $x=\lim_{n\rightarrow\infty }P_n(x)$. Therefore, passing to the limit in \eqref{ine lim}, we conclude that $\|G_{\pi,m}^\vee(x)\|\leq C\|x\|$ for all $m\in\N$, as desired.


\end{proof}

In order to prove \Cref{u-greedy theorem intro}, we need the following lemma.
\begin{lemma}\label{greedy lemma}
Suppose that (v) in \Cref{equivalence of constants.} fails. Then for every $C\geq 1$ and every finite set $A\subseteq \N$, there exists $x\in E$ with $|supp(x)|<\infty$ and $supp(x)\cap A=\emptyset$ such that $\|\mathcal{G}_{|\text{supp}(x)|}^\vee(x)\|>C\|x\|$.
\end{lemma}
\begin{proof}



Fix $C$ and $A$ as above and define $M:=\sum_{n\in A}\|e_n^*\|\|e_n\|$. By assumption, there exists a finitely supported $y\in E$ such that
 \begin{equation}\|
 {\mathcal{G}^\vee}_{|\text{supp}(y)|}(y)\|>(C(1+M)+M)\|y\|.
\end{equation}  
Set $r=|\text{supp}(y)|$ and define $x=(I-P_A)(y)$, where for $z\in E$, $P_A(z):=\sum_{n\in A}e_n^*(z)e_n$. 
Since $x\in E$ has finite support disjoint from $A$, we  must simply show that $\|\mathcal{G}_m^\vee(x)\|>C\|x\|$, where $m$ is the cardinality of the support of $x$. For this, notice that $\mathcal{G}_r^\vee(y)\leq \mathcal{G}_m^\vee(x)+\sum_{k\in A}|e_k^*(y)e_k|$, so that 

$$\|\mathcal{G}_m^\vee(x)\|\geq \|\mathcal{G}_r^\vee(y)\|-\|\sum_{k\in A}|e_k^*(y)e_k|\|>(C(1+M)+M)\|y\|-M\|y\|.$$
Inserting the inequality $\|x\|\leq (1+M)\|y\|$ into the above, we conclude that $\|\mathcal{G}_m^\vee(x)\|>C\|x\|$, as desired.
\end{proof}
We may now  establish the greedy analogue of  \Cref{bidecomposition}.
\begin{theorem}\label{u-greedy theorem.}
For a semi-normalized basic sequence $\mathcal{B}=(e_n)$ in a Banach lattice $X$, the following are equivalent.
\begin{enumerate}
\item For all $x\in E$,  $\mathcal{G}_m(x)\xrightarrow{u}x$;
\item For all $x\in E$, $\mathcal{G}_m(x)\xrightarrow{o}x$;
\item For all $x\in E$, $|\mathcal{G}_m(x)|\leq u$ for some $u\in X$ and all $m$;
\item For all $x\in E$, $(\bigvee_{n=1}^m|\mathcal{G}_n(x)|)_m$ is norm bounded;
\item There exists $C\geq 1$ such that for all $x\in E$ and $m\in \mathbb{N}$,
$$\bigg\|\bigvee_{n=1}^m\left|\mathcal{G}_n(x)\right|\bigg\|\leq C\|x\|.$$
\end{enumerate}

\end{theorem}
\begin{proof}
It is clear that (i)$\Rightarrow$(ii)$\Rightarrow$(iii)$\Rightarrow$(iv).
\medskip

 (iii)$\Rightarrow$(v): Suppose not. Then using \Cref{greedy lemma} we get a sequence $(x_k)$ of elements of $E$ such that for each $k$:
\begin{enumerate}
\item[(a)] $m_k:=|supp(x_k)|$ is finite and $supp(x_k)\cap supp(x_i)=\emptyset$ for $i=1,\dots, k-1$;
\item[(b)] $\|x_k\|\leq 2^{-k}$;
\item[(c)] $\|\mathcal{G}_{m_k}^\vee(x_k)\|\geq k$;
\item[(d)] $\max \{|e_n^*(x_k)| :n\in \N \}<\min \{|e_n^*(x_{k-1})| : n\in supp(x_{k-1})\}.$
\end{enumerate}
Indeed, we can certainly do this for $k=1$, so suppose that we have constructed $x_1,\dots, x_{k-1}$. Let $\mu=\min\{|e_n^*(x_{k-1})|: n\in supp(x_{k-1})\}$ and put $C_k=\max\{k2^k, 2Kk\mu^{-1}\}$ where $K=\sup_n\|e_n^*\|$. Using \Cref{greedy lemma} there exists $x_k$ of finite support disjoint from $\cup_{i=1}^{k-1}supp(x_i)$  such that $\|\mathcal{G}_{m_k}^\vee(x_k)\|>C_k\|x_k\|$. Scaling $x_k$ we can take $\|x_k\|=kC_k^{-1}\leq 2^{-k}$, so that $\|\mathcal{G}_{m_k}^\vee(x_k)\|\geq k$.  Then for every $n\in \N$ we have
$$|e_n^*(x_k)|\leq \|e_n^*\|\|x_k\|<\mu.$$
Now, the series $\sum_{k=1}^\infty x_k$ converges to some $x\in E$. Note by construction that $\sum_{k=1}^{j-1}x_k=\mathcal{G}_{l_j}(x)$ is a strictly greedy sum of $x$ (here $l_j:=\sum_{k=1}^{j-1}m_k$, $j\geq 2$, and $l_1:=0$). Write $x_j$ in its natural greedy ordering as $e_{\rho_j(1)}^*(x_j)e_{\rho_j(1)}+\cdots + e_{\rho_j(m_j)}^*(x_j)e_{\rho_j(m_j)}$ and notice that $\mathcal{G}_{l_j}(x)+e_{\rho_j(1)}^*(x_j)e_{\rho_j(1)}+\cdots +e_{\rho_j(r)}^*(x_j)e_{\rho_j(r)}$ is a natural greedy sum of $x$  for each $1\leq r\leq m_j$. Hence, by assumption, there is a $u$ with $|\mathcal{G}_{l_j}(x)+e_{\rho_j(1)}^*(x_j)e_{\rho_j(1)}+\cdots +e_{\rho_j(r)}^*(x_j)e_{\rho_j(r)}|\leq u$. Note that $u$ is uniform (independent of $j$ and $r$) since we are building a natural greedy approximation of $x$, i.e., every term we are bounding is a greedy sum of the same approximation. In particular, $|\mathcal{G}_{l_j}(x)|\leq u$ for all $j$, hence $|e_{\rho_j(1)}^*(x_j)e_{\rho_j(1)}+\cdots +e_{\rho_j(r)}^*(x_j)e_{\rho_j(r)}|\leq 2u$. Taking sup we get $\mathcal{G}_{m_j}^\vee(x_j)\leq 2u$, so that $j\leq \|\mathcal{G}_{m_j}^\vee(x_j)\|\leq 2\|u\|$. This is a contradiction. 
\medskip

(v)$\Rightarrow$(i):  By \cite[Theorem 10.2.3, Lemma 10.2.5]{AK16} we get that for each $y\in E$, $\mathcal{G}_m(y)\xrightarrow{\|\cdot\|}y$. Fix $x\in E$.  Since $\mathcal{G}_m(x)\xrightarrow{\|\cdot\|}x$ there exists $m_1<m_2<\dots$ with $\mathcal{G}_{m_k}(x)\xrightarrow{u}x$.  Passing to a further subsequence and using that $(\mathcal{G}_m(x))$ is Cauchy, we may assume that for all $i>m_k$, $\|\mathcal{G}_i(x)-\mathcal{G}_{m_k}(x) \|<\frac{1}{2^k}$. In particular, $\|x-\mathcal{G}_{m_k}(x)\|\leq \frac{1}{2^k}$.
\medskip

Consider the element $y_k=x-\mathcal{G}_{m_k}(x)$. Then for each $i>m_k$, $\mathcal{G}_i(x)-\mathcal{G}_{m_k}(x)$ is just $\mathcal{G}_{i-m_k}(y_k)$. Hence, our assumption yields that $$u_k=\bigvee_{i=m_k+1}^{i=m_{k+1}}\left|\mathcal{G}_i(x)-\mathcal{G}_{m_k}(x)\right|$$ has norm at most $C\|y_k\|\leq \frac{C}{2^k}$. Define $e:=\sum_{k=1}^\infty ku_k$. Then for every $k\in \mathbb{N}$, $u_k\leq \frac{e}{k}$, so that $u_k\xrightarrow{u}0$.
\medskip

Since $|\mathcal{G}_{m_k}(x)-x|+u_k\xrightarrow{u}0$ there exists a vector $f_1>0$ with the property that for any $\varepsilon>0$ there exists $k^*$, for any $k\geq k^*$, $|\mathcal{G}_{m_k}(x)~-~x|+u_k\leq \varepsilon f_1$. Fix $\varepsilon$, and find the required $k^*$. Let $i \in \mathbb{N}$ with $i> m_{k^*}$. We can find $k\geq k^*$ with $m_k<~i\leq ~m_{k+1}$ so that
$$|\mathcal{G}_{i}(x)-x|\leq |\mathcal{G}_{m_k}(x)-x|+|\mathcal{G}_{i}(x)-\mathcal{G}_{m_k}(x)|=|\mathcal{G}_{m_k}(x)-x|+u_k\leq \varepsilon f_1.$$
This shows that $\mathcal{G}_{i}(x)\xrightarrow{u}x$, and hence that $\mathcal{B}$ satisfies (i).
\medskip

We have thus established that  (i)$\Leftrightarrow$(ii)$\Leftrightarrow$(iii)$\Leftrightarrow$(v)$\Rightarrow$(iv). For (iv)$\Rightarrow$(i) assume that for each $x\in E$, $(\mathcal{G}_m^\vee(x))$ is norm bounded. Since $(\mathcal{G}_m^\vee(x))$ is an increasing sequence, it has a supremum in $X^{**}$, hence $(\mathcal{G}_m(x))$ is order bounded in $X^{**}$. Thus, (ii) holds in $X^{**}$, so (i) holds in $X^{**}$. However, uniform convergence of sequences passes freely from $X^{**}$ to $X$ by \cite[Proposition 2.12]{TT19}. Consequently, (i) holds in $X$.
%
\end{proof}

We will call a semi-normalized basic sequence $(e_n)$ in a Banach lattice $X$ \emph{uniformly quasi-greedy} if it satisfies any, and hence all, of the conditions in \Cref{u-greedy theorem.}. Note that the last condition in \Cref{u-greedy theorem.}   says that $C_{qg}^\vee:=\sup_m\|\mathcal{G}_m^\vee\|<\infty$.  We will call $C_{qg}^\vee$ the \term{uniform quasi-greedy constant of} $(e_n)$. 

\begin{example}\textit{Wavelets:}
By controlling the maximal function of the greedy approximations by the Hardy-Littlewood maximal function, Tao (\cite{Tao96}, see also \cite{KO97})  proved that  all compactly supported wavelet bases (such as the Haar and Daubechies wavelets) are uniformly quasi-greedy in $L_p$ for all $1<p<\infty$. In particular, considering a rearrangement of the Haar basis that is not bibasic (\cite[Example 6.2]{TT19}), we see that uniformly quasi-greedy bases need not be bibases.
\end{example}
\begin{example}\textit{The trigonometric and Walsh bases:} In \cite{K96}, K\"orner showed  the existence of a real-valued function $f\in L_2(\mathbb{T})$ whose greedy algorithm with respect to the trigonometric system diverges almost everywhere. In particular, this provides an example of a basis that is both greedy and a bibasis but is not uniformly quasi-greedy. Similarly, in \cite{K06} it is shown that the greedy algorithm for the Walsh basis need not converge almost everywhere. 
\end{example}

\begin{remark}\label{recover usual}
As mentioned in \Cref{intro sec}, due to the freedom in choosing the ambient Banach lattice $X$, the theory we will develop in this article for uniformly quasi-greedy bases  will recover and extend the usual theory for quasi-greedy bases. To see this, let $(e_n)$ be a semi-normalized Schauder basis of a Banach space $E$. As $E$ is separable, we may consider $E$ as a subspace of $C[0,1]$. In $C[0,1]$, it is easy to see that uniform convergence agrees with norm convergence and $\| |x|\vee |y|\|=\|x\|\vee\|y\|$ for all $x,y\in C[0,1]$. Thus, choosing $X=C[0,1]$ in \Cref{u-greedy theorem.}, we recover the classical characterization of quasi-greedy bases in \Cref{W00}. Similarly, when $X=C[0,1]$,  our results in \Cref{Sub3} will recover the classical structural properties of quasi-greedy bases. This leads to two different perspectives on \Cref{u-greedy theorem.}:
\medskip

First, suppose that the Banach lattice $X$ is fixed. Then one would be interested in determining which  quasi-greedy bases (or quasi-greedy basic sequences) are uniformly quasi-greedy in $X$, as these bases would exhibit  the enhanced approximation properties in \Cref{u-greedy theorem.}.
On the other hand, suppose  that the Banach space $E$ -- or even a  basis $(e_n)$ of $E$  --  is fixed, and consider an embedding of $E$  into a Banach lattice $X$. If one chooses $X$ to be $C[0,1]$, then several of the theorems we will prove in this paper will recover those from the theory of quasi-greedy bases. However, if one chooses $X$ differently, then the sequence $(e_n)$ may, or may not, be uniformly quasi-greedy with respect to $X$. Note that the property of being uniformly quasi-greedy is ``stable" with respect to the choice of $X$, in the sense that if $(e_n)$ is a basic sequence in $X$, and $X$ is a closed sublattice of $Y$, then $(e_n)$ will be uniformly quasi-greedy with respect to $X$ iff it is uniformly quasi-greedy with respect to $Y$. In the next section, we will show that if one puts conditions on $X$ (say, order continuity, or being a classical Banach lattice) or on how $E$ sits in $X$ (say, complementably, or even $E=X$) then one can prove theorems for bibasic and uniformly quasi-greedy bases that have no analogues for Schauder or quasi-greedy bases. 
\end{remark}

\begin{remark}\label{Recall FDD}
\Cref{u-greedy theorem.}  is equally valid for Schauder decompositions. Since greedy decompositions do not seem to appear in the literature, we quickly sketch what we mean:
\medskip

Let $\mathcal{B}=(E_n)$ be a sequence of closed non-zero subspaces of a Banach lattice $X$ which forms a Schauder  decomposition of $E:=[E_n].$  Fix $x=\sum_{n=1}^\infty e_n \in E$, with $e_n\in E_n$ for all $n$. A \term{greedy ordering} of $x$ is an injective map $\pi : \mathbb{N}\to \mathbb{N}$ such that $\{n: e_n\neq 0\}\subseteq \pi(\mathbb{N})$ and  $(\|e_{\pi(n)}\|)$ is non-increasing. The \term{$m$-th greedy sum} of $x$ associated to $\pi$ is given by $G_{\pi,m}(x):=\sum_{n=1}^m e_{\pi(n)},$
and the sequence $(G_{\pi,m}(x))_{m=1}^\infty$ is called a \term{greedy approximation} of $x$.  $G_{\pi,m}(x)$ is said to be a \term{strictly greedy sum} of $x$ of order $m$ if $G_{\pi,m}(x)=G_{\sigma,m}(x)$ for all greedy orderings $\sigma$.  The \term{$m$-th  natural greedy sum} of $x$ is $\mathcal{G}_m(x):=\sum_{n=1}^m e_{\rho(n)}$ where $\rho: \mathbb{N}\to \mathbb{N}$ is such that $\{n: e_n\neq 0\}\subseteq \rho(\mathbb{N})$ and if $j<k$ then either $\|e_{\rho(j)}\|>\|e_{\rho(k)}\|$ or $\|e_{\rho(j)}\|=\|e_{\rho(k)}\|$ and $\rho(j)<\rho(k)$. A \term{quasi-greedy decomposition} is a decomposition for which $\mathcal{G}_m(x)\xrightarrow{\|\cdot\|}x$ for all $x\in E$. Defining $\mathcal{G}_m^\vee$ in the obvious way, the proof of the following theorem is almost identical to the semi-normalized basic sequence case.
\end{remark}
\begin{theorem}\label{u-greedy theorem. decompo}
Let $\mathcal{B}=(E_n)$ be a sequence of subspaces of a Banach lattice $X$ which forms a Schauder decomposition of $[E_n]$.  The following are equivalent.
\begin{enumerate}
\item For all $x\in E$,  $\mathcal{G}_m(x)\xrightarrow{u}x$;
\item For all $x\in E$, $\mathcal{G}_m(x)\xrightarrow{o}x$;
\item For all $x\in E$, $|\mathcal{G}_m(x)|\leq u$ for some $u\in X$ and all $m$;
\item For all $x\in E$, $(\bigvee_{n=1}^m|\mathcal{G}_n(x)|)_{m}$ is norm bounded;
\item There exists $C\geq 1$ such that for all $x\in E$ and $m\in \N$,
$$\bigg\|\bigvee_{n=1}^m\left|\mathcal{G}_n(x)\right|\bigg\|\leq C\|x\|.$$
\end{enumerate}

\end{theorem}
In particular, arguing as in \Cref{recover usual}, the ``usual" characterization of quasi-greedy bases in \Cref{W00} is valid for quasi-greedy decompositions -- we will focus on the basis case for convenience of the reader; however, we will return to the subject of finite dimensional decompositions (FDDs) at various points in the paper. 
\medskip

We also mention that, although stated for the natural ordering, being uniformly quasi-greedy is independent of the greedy ordering. Indeed, noting that absolutely convergent series converge uniformly, the proof from \cite[Lemma 10.2.5]{AK16} can be used to show the following.

\begin{proposition}\label{U-independ}
Let $(e_n)$ be a basic sequence in a Banach lattice $X$. The following are equivalent.
\begin{enumerate}
\item $G_m(x)\xrightarrow{u}x$ for every $x\in E$ and every greedy approximation $(G_m(x))$;
\item $(e_n)$ is uniformly quasi-greedy;
\item For every $x\in E$ there is a greedy approximation $(G_m(x))$ such that $G_m(x)\xrightarrow{u}x$;
\item For every $x\in E$ with infinite support its strictly greedy approximation converges uniformly to $x$.

\end{enumerate}
\end{proposition}
Results analogous to \Cref{equivalence of constants.} and \Cref{U-independ} are also valid for conditions (ii), (iii) and (iv) of \Cref{u-greedy theorem.}. For example, one can show that being uniformly quasi-greedy is equivalent to the fact that for every $x\in E$ and greedy ordering $\pi$, there is a $u^\pi\in X$ satisfying that for all $m$, $|G_{\pi,m}(x)|\leq u^\pi$. However, the next example shows that one cannot pick $u^\pi$ independently of $\pi$:
\begin{example} \label{Not uniform} We will construct a uniformly quasi-greedy sequence $(e_k)$ in the Banach lattice $X=\ell_2(L_2)$ such that there exists $x\in E$ for which there is no $u\in X$ such that for every $m\in\N$ and every greedy ordering $\pi$ we have $|G_{\pi,m}(x)|\leq u$.  In the $n$-th copy of $L_2$ we consider the first $n^2$ Rademacher vectors. Order these into a normalized  basic sequence $(e_k)$ in lexicographical order; by standard martingale inequalities the Rademacher's are bibasic in every ordering, from which it follows that $(e_k)$ is uniformly quasi-greedy. Now  consider the vector $x=\sum_{k=1}^\infty a_ke_k\in \ell_2(L_2)$ defined by having all of the $a_k$ in the $n$-th copy of $L_2$ being $\frac{1}{n^2}$. Note that this series converges as $(e_k)$ is an orthonormal sequence, and it is already in its natural greedy ordering. 
Next, recall the inequality $\sup_{\delta_n=0,1}|\sum_{n=1}^m\delta_nx_n|\geq \frac{1}{2}\sum_{n=1}^m|x_n|$ and suppose that $u\in X$ dominates $|G_{\pi,m}(x)|$ for each $\pi$ and $m$. Then in the $n$-th block, $u$ must dominate the partial sums of each permutation of the Rademacher's, hence it must dominate half of the sum of their absolute values. This means that the $n$-th block contributes at least $\frac{1}{2}$ to the norm of $u$. Since each of the blocks are disjoint, this forces $\|u\|$ to be arbitrarily large, a contradiction.
\end{example}
\begin{remark}
The motivation for \Cref{Not uniform} stems from the various characterizations of absolute sequences given in \cite{TT19}. Suppose that $(x_k)$ is an unconditional basic sequence in a Banach lattice $X$. Then $(x_k)$ is said to be \emph{permutable} if it is bibasic in every ordering; that is, if for each permutation $\sigma$ and $x\in [x_k]$ there exists $u^\sigma$ such that $|P^\sigma_nx|\leq u^\sigma$ for all $n$, where $P^\sigma_n$ is the $n$-th partial sum associated to the permutation $\sigma$. It turns out that one can choose $u^\sigma$ independently of $\sigma$ if and only if $(x_k)$ is absolute -- see \cite[Proposition 7.5]{TT19}.  As shown in \Cref{Not uniform}, the property that the order bound $u^\pi$ in statement (iii) of \Cref{u-greedy theorem.} can be chosen independently of the greedy ordering $\pi$ can be viewed as an intermediate between being uniformly quasi-greedy and absolute (see \cite[Theorem 7.2]{TT19} for further characterizations of absolute sequences). 
\end{remark}
\begin{remark}
    Apart from \Cref{W00}, the other main characterization in the theory of greedy bases is the equivalence between being greedy and being unconditional and democratic. This characterization can also be generalized to the Banach lattice setting as follows: Let $(M,\Sigma,\mu)$ be a measure space and let $X$ be a Banach function space over $(M,\Sigma,\mu)$.   For each $x\in X$ we let $G_{x}\in\Sigma$ denote any finite measure  set such that
$$ \text{ess sup}_{s\in G^c_{x}}|x(s)|\leq \text{ess inf}_{t\in G_{x}}|x(t)|<\infty.$$
 For each $A\in\Sigma$, we let $P_A$ be the restriction operator to $A$.  We say that $X$ is \emph{greedy} if there exists some constant $C>0$ such that for all $x\in X$ and all sets of the form $G_{x}$ we have  
\begin{equation*}\label{greedy}
\|x-P_{G_{x}}x\|\leq C \inf_{\mu(A)\leq\mu(G_x);y\in P_A X}\|x-y\|.
\end{equation*}
We say that $X$ is {\em democratic} if there exists $D>0$ such that $\|1_A\|\leq D\|1_B\|$ for all $A,B\in\Sigma$ with $\mu(A)\leq  \mu(B)<\infty$. 
Routine arguments then show that $X$ is greedy if and only if it is democratic, which extends the characterization of (weight)-greedy bases from the discrete setting where $\mu$ is atomic to the continuous setting of Banach function spaces.
\end{remark}


\subsection{The bibasis theorem fails for frames}\label{frame subsec}
One of the surprising features of \Cref{bidecomposition} is that although order convergence and order boundedness generally do not pass between sublattices, statement (iv) is stable under passing between sublattices, in the sense that if $(x_k)\subseteq X\subseteq Y$ then (iv) holds in $X$ iff it holds in $Y$. In statement (ii) of \Cref{bidecomposition} one actually has a simultaneously norm and order convergent sequence, so one may wonder if these sequences pass freely between closed sublattices. The next result says that the answer is no; we will later expand on this example to show that the bibasis theorem fails for frames.
\begin{example}\label{no pass}
Let $\varphi$ be an Orlicz function and for each sequence $x=(x_k)$ of real numbers, define  $I(x)=\sum_{k=1}^\infty \varphi (|x_k|)$. Then $\ell_{\varphi}:=\{x : \exists \lambda>0, \  I(\lambda x)\leq 1\}$ is a Banach lattice under the Luxemburg norm, and $h_{\varphi}:=\{x : \forall \lambda>0, \ I(\lambda x)<\infty\}$ is a closed order dense ideal of $\ell_\varphi$. It is well-known that $\ell_\varphi=h_\varphi$ iff $\varphi$ satisfies the $\Delta_2$-condition near zero, so choose any $\varphi$ that fails the $\Delta_2$-condition near zero. Fix some positive vector $x=(x_k)$ in $\ell_\varphi \setminus h_\varphi$ and define $y^k=x_ke^k\in h_\varphi$, where $e^k$ is the $k$-th standard unit vector. Next, define $z^k$ by $z^k_n=x_n$ if $n\geq k$, and $0$ otherwise. Then $0\leq y^k\leq z^k\downarrow 0$ in $\ell_{\varphi}$, so that $y^k\xrightarrow{o}0$ in $\ell_{\varphi}$. Note that since $x\in \ell_{\varphi}$, $(y^k)$ is norm null, but $(y^k)$ is not even order bounded in $h_\varphi$, since any upper bound for $(y^k)$ must coordinate-wise dominate $x\not \in h_\varphi$.

\end{example}
We are now ready to show that \Cref{bidecomposition} fails for frames. Recall that, given  a Banach space $E$,  a sequence $(x_k,f_k)\in E\times E^*$ is a \emph{Schauder frame} for $E$ if for all $x\in E$ we have $x=\sum_{k=1}^\infty f_k(x)x_k$. Note that each of the statements in \Cref{bidecomposition} has an obvious  ``frame" analogue. Specifically, for a frame $(x_k,f_k)$ with $(x_k)$ contained in a Banach lattice,  we can consider the following properties:
\begin{enumerate}
\item For all $x\in [x_k], \sum_{k=1}^nf_k(x)x_k\xrightarrow{u}x$;
\item For all $x\in [x_k], \sum_{k=1}^nf_k(x)x_k\xrightarrow{o}x$;
\item For all $x\in [x_k], \left(\sum_{k=1}^nf_k(x)x_k\right)_n$ is order bounded;
\item  For all $x\in [x_k], \left(\bigvee_{n=1}^m\left|\sum_{k=1}^nf_k(x)x_k\right|\right)_m$ is norm bounded;
\item There exists $M\geq 1$ such that for all $x\in [x_k]$ and  $m\in \mathbb{N},$ $$\bigg\|\bigvee_{n=1}^m\left|\sum_{k=1}^nf_k(x)x_k\right|\bigg\|\leq M\|x\|.$$
\end{enumerate}
It is easy to see that (i)$\Rightarrow$(ii)$\Rightarrow$(iii)$\Rightarrow$(iv)$\Leftrightarrow$(v). However, it turns out that the reverse implications do not hold, in general. To see that (iv)$\not\Rightarrow$(iii), we will modify \Cref{no pass}. However, before  doing so, we give an easier example which shows that  (iii)$\not\Rightarrow$(ii). 
\begin{example}\label{example}
Let $X=L_p[0,1]$, $1<p<\infty$. Since $X$ is reflexive, order convergence agrees with uniform convergence (see \cite{Bed-Wirth}) and increasing norm bounded sequences have supremum. 
Hence, for any frame of $X$, properties (i) and (ii) coincide, as do (iii), (iv) and (v). We will show that (iii)$\not\Rightarrow$(ii). For this, we let $(t_n)$ be the ``typewriter" sequence, $t_n=\chi_{[\frac{n-2^k}{2^k},\frac{n-2^k+1}{2^k}]}$, where $k\geq 0$ is such that $2^k\leq n<2^{k+1}$. Let $(h_n,h_n^*)$ denote the Haar basis with its coordinate functionals and choose $f\in X^*$ with $f(\one)=1$. We define our sequence $(x_k,f_k)$ by weaving the typewriter sequence through the Haar basis as follows: $$(x_k)=(h_1,t_1,-t_1,h_2,t_2,-t_2,h_3,t_3,-t_3,\dots),$$ $$(f_k)=(h_1^*,f,f,h_2^*,f,f,h_3^*,f,f,\dots).$$
It is  easy to see that $(x_k,f_k)$ is a frame. Moreover, since $(h_k)$ is a bibasis, it is easy to deduce that for every $x\in X$, $P_nx:=\sum_{k=1}^n f_k(x)x_k$ defines an order bounded sequence. However, $P_n(\one)\not\xrightarrow{a.e.}\one$.
\end{example}
\begin{remark}
    A slight modification of the above construction (by instead weaving through a weakly null normalized sequence) allows one to give examples of reproducing pairs that are not Schauder frames. This solves an open problem from \cite{hart2023overcomplete}.
\end{remark}
As we have already remarked, the implication (iv)$\Rightarrow$(iii) above also fails for frames; however, there is a significant strengthening of this fact. Indeed, recall that a frame $(x_k,f_k)$ for a closed subspace $E$ of a Banach lattice $X$ is \term{absolute} if for each $x\in E$,
$$\sup_n \bigg\|\sum_{k=1}^n\left|f_k(x)x_k\right|\bigg\|<\infty.$$
Baire Category arguments then give the existence of some $A\geq 1$ with  
$$ \bigg\|\sum_{k=1}^n\left|f_k(x)x_k\right|\bigg\|\leq A\|x\|$$ for each $x\in E$. Clearly, absolute frames satisfy property (v). However, as we will now show, they need not satisfy property (iii).
\begin{example}\label{Mon is nec}
Consider an Orlicz sequence space $\ell_{\varphi}$ failing the $\Delta_2$-condition near zero. The unit vectors $(e_k)$ are disjoint and hence absolute. We will view  $(e_k)\subseteq h_\varphi \subseteq \ell_{\varphi}$. As in \Cref{no pass}, we choose $0\leq x=(x^k)$ in $\ell_\varphi \setminus h_\varphi$ and define $y_k=x^ke_k$.
Let $f$ be a non-zero functional. We  consider the sequence
$$(e_1, y_1,-y_1, e_2, y_2, -y_2, e_3,\dots)$$
equipped with the functionals
$$(e_1^{*}, f, f, e_2^*, f, f, e_3^*,\dots).$$
For $u\in h_{\varphi}$, the partial sums  of this sequence will look like 
$\sum_{k=1}^n e_k^*(u)e_k$ or $\sum_{k=1}^n e_k^*(u)e_k +f(u)x^ne_n$, and so this sequence will be a frame.
Also, if we sum moduli we can bound the norm by $$\|\sum_{k=1}^\infty |e_k^*(u)e_k|\|+2|f(u)|\|\sum_{n=1}^\infty x^ke_k\|<\infty.$$ Hence, we have constructed an absolute frame.
\medskip

Now, let $u$ be such that $f(u)\neq 0$. If the partial sums for $u$ were order bounded in $h_{\varphi}$, we would deduce that the sequence $(x^ne_n)$ is order bounded in $h_{\varphi}$, which we showed in \Cref{no pass} is not true.
\end{example}

The main reason why the above example works is that the space $X=h_\varphi$ is not $\sigma$-monotonically complete. Recall that a Banach lattice $X$ is \emph{$\sigma$-monotonically complete} if  all positive increasing norm  bounded sequences in $X$ have supremum. In this case, one can deduce order convergence of the partial sums for absolute frames.
\begin{proposition}\label{mc}
If $X$ is $\sigma$-monotonically complete and $(x_k,f_k)$ is an absolute frame for $E\subseteq X$ then for each $x\in E$, $\sum_{k=1}^\infty f_k(x)x_k\xrightarrow{o}x$. In particular, each absolute frame has order convergent expansions, when the order convergence in $X^{**}$ is used.
\end{proposition}
\begin{proof}
Fix $x\in E$. Since $(x_k,f_k)$ is absolute, the sequence $(\sum_{k=1}^n|f_k(x)x_k|)$ is increasing and norm bounded, so has a supremum $u$. It is then easy to see that $|x-\sum_{k=1}^nf_k(x)x_k|\leq u-\sum_{k=1}^{n}|f_k(x)x_k|\downarrow 0$. To finish the proof, recall the standard fact that if $Y$ is a Banach lattice then $Y^*$ is monotonically complete. 
\end{proof}
\Cref{mc} shows that absolute frames  admit a reconstruction formula with respect to an appropriate order convergence. However, being an absolute frame is a very strong property. Indeed, by modifying the proof of \cite[Theorem 3]{JS}, one  obtains the following result.
\begin{proposition}
Let $(x_k,f_k)$ be an absolute frame in $L_p[0,1]$, $1\leq p<\infty$. Then $[x_k]$ embeds isomorphically into $\ell_p$.
\end{proposition}
On the other hand, although property (i) is the strongest of the five properties listed above, we will now show that frames satisfying this condition are in complete abundance. To set the stage,  let $E$ be a closed subspace of a Banach lattice $X$ and recall that a sequence $(x_k,f_k)$ in $E\times E^*$ is  a  \emph{$u$-frame} if for each $x\in E$ we have $\sum_{k=1}^n f_k(x)x_k\xrightarrow{u}x$. The following result shows that if $E$ admits a $u$-frame (which happens, in particular, if $E$ has a bibasis) then one can construct a $u$-frame for $E$ with the $(x_k)$ lying in any subset of $E$ with dense span.
\begin{proposition}
    Let $E$ be a closed subspace of a Banach lattice $X$ and let $M\subseteq E$ have dense span. If $E$ admits a $u$-frame then there is a $u$-frame $(x_k,f_k)$ for $E$ with each $x_k\in M$.
\end{proposition}
\begin{proof}
    This result was proven in \cite[Theorem 4.4]{FPT2} under the assumption that $E=X$ and $(x_k,f_k)$ is a bibasis. Therefore, it suffices to extend the result by only assuming that $(x_k,f_k)$ is a $u$-frame for a subspace of $X$. The key to doing this is to prove a perturbation result; namely, that if $(x_k,f_k)$ is a $u$-frame for a closed subspace $E$ of a Banach lattice $X$ and $0<\varepsilon<1$ then any sequence $(y_k)$ in $E$ for which
$$\|x_k-y_k\| \leq \frac{\varepsilon}{2^{2k+1}\|f_k\|}$$
can be paired with a sequence $g_k\in E^*$ so that $(y_k,g_k)$ is a $u$-frame for $E$. Such a perturbation result allows one to construct, for any  dense set $M$ of $E$, a $u$-frame for $E$ consisting  of elements of $M$. After this, the proof of \cite[Theorem 4.4]{FPT2} can be used to replace the assumption that $M$ is dense with the assumption that $M$ has dense span. To see that the above perturbation result holds, let $x\in E$ and define $S(x)=\sum_{k=1}^{\infty}f_k(x)(x_k-y_k)$. It is shown in \cite[Lemma 2.3]{PS} that $S$ is well-defined and $\|S\|<1$. In fact, it is easy to see that the sum is uniformly Cauchy and hence uniformly converges. One then defines $T=I-S$, so that $T(x)=u-\sum_{k=1}^\infty f_k(x)y_k$. Replacing $x$ with $T^{-1}x$, we see that $x=u-\sum_{k=1}^\infty f_k(T^{-1}x) y_k$, so that $(y_k, (T^{-1})^*f_k)$ is a $u$-frame for $E$.
\end{proof}
\begin{remark}
  Interesting examples of Gabor frames exhibiting enhanced convergence properties can be found in \cite{MR1817671}; see also \cite{MR1836633} for nonharmonic Fourier series.
\end{remark}

\subsection{Properties of uniformly quasi-greedy bases}\label{Sub3}
 As explained in \Cref{recover usual}, any result about quasi-greedy bases can be viewed as a result about uniformly quasi-greedy bases with ambient Banach lattice   $C[0,1]$.  In this subsection, we record some extensions of the fundamental results on quasi-greedy bases   to other ambient Banach lattices $X$. 
 \medskip
 
 We begin by generalizing the well-known fact that quasi-greedy bases are unconditional for constant coefficients.
\begin{proposition}\label{almost unc}
Suppose that $(e_n)$ is a uniformly quasi-greedy basic sequence. Then the following hold.
\begin{enumerate}
\item For any distinct indices $n_1,\dots, n_m$ we have
$$\bigg\|\bigvee_{k=1}^m\left|\sum_{i=1}^ke_{n_{i}}\right|\bigg\|\leq C_{qg}^\vee\bigg\|\sum_{i=1}^me_{n_i}\bigg\|.$$

\item $(e_n)$ is permutable for constant coefficients, i.e., for any distinct indices $n_1,\dots, n_m$ and any choices of signs $\varepsilon_i$ we have 
$$(2C_{qg}^\vee )^{-1}\bigg\|\bigvee_{k=1}^m\left|\sum_{i=1}^k e_{n_i}\right|\bigg\|\leq \bigg\|\sum_{i=1}^m\varepsilon_i e_{n_i}\bigg\|\leq \bigg\|\bigvee_{k=1}^m\left|\sum_{i=1}^k \varepsilon_ie_{n_i}\right|\bigg\|\leq 2C_{qg}^\vee \bigg\|\sum_{i=1}^me_{n_i}\bigg\|.$$
\end{enumerate}
\end{proposition}
\begin{proof}
(i) follows since $(\sum_{i=1}^ke_{n_{i}})$ is a greedy approximation of $x=~\sum_{i=1}^m e_{n_i}$.
For the right inequality in (ii), notice that
$$\left|\sum_{i=1}^k\varepsilon_i e_{n_i}\right|\leq \left|\sum_{\{i\leq k\ : \  \varepsilon_i=1\}} \varepsilon_{i}e_{n_i}\right|+\left|\sum_{\{i\leq k\ :\  \varepsilon_i=-1\}}\varepsilon_ie_{n_i}\right|,$$
where if we happen to sum over the empty set, we declare the sum to be $0$. It follows that
$$\bigg\|\bigvee_{k=1}^m\left|\sum_{i=1}^k\varepsilon_i e_{n_i}\right|\bigg\|\leq \bigg\|\bigvee_{k=1}^m\left|\sum_{\{i\leq k\ : \  \varepsilon_i=1\}} \varepsilon_{i}e_{n_i}\right|\bigg\|+\bigg\|\bigvee_{k=1}^m\left|\sum_{\{i\leq k\ :\  \varepsilon_i=-1\}}\varepsilon_ie_{n_i}\right|\bigg\|.$$
In this form,   we can ignore the $\varepsilon_i$ and interpret the above terms as greedy sums of $x$, which gives us our desired estimate.
\medskip



To obtain the first inequality in (ii), let $A=\{i: \varepsilon_i=1\}$ and $B=\{1,\dots,m\}\setminus A$. Then a generic term $\sum_{i=1}^ke_{n_i}$ can be decomposed as follows:
$$\left|\sum_{i=1}^ke_{n_i}\right|\leq \left|\sum_{A\cap\{1,\dots,k\}}e_{n_i}\right|+\left|\sum_{B\cap\{1,\dots,k\}}(-1)e_{n_i}\right|.$$
Taking sup over the above inequality and using that both the first and second terms on the right-hand side form part of a greedy approximation of $\sum_{i=1}^m\varepsilon_ie_{n_i}$, we obtain our desired estimate with constant  $(2C_{qg}^\vee )^{-1}$.




\end{proof}
\begin{corollary}\label{Convex max}
Suppose that $(e_n)$ is a uniformly quasi-greedy basic sequence. Then for any distinct indices $n_1,\dots, n_m$ and any real numbers $a_1,\dots, a_m$, we have 
$$\bigg\|\bigvee_{k=1}^m\left|\sum_{i=1}^k a_ie_{n_i}\right|\bigg\|\leq 2\max |a_i|C_{qg}^\vee \bigg\|\sum_{i=1}^m e_{n_i}\bigg\|.$$
\end{corollary}
\begin{proof}
Without loss of generality, we may assume that $\max |a_i|=1$. By \cite[Theorem 3.13]{Heil10}, there exists $c_l\geq 0$ and signs $\varepsilon_l^i$, $l=1,\dots, m+1$, $i=1,\dots, m$ such that $\sum_{l=1}^{m+1}c_l=1$ and $\sum_{l=1}^{m+1}\varepsilon_l^i c_l=a_i$ for $i=1,\dots, m$. It follows that
\begin{eqnarray*}
    \bigg\|\bigvee_{k=1}^m\left|\sum_{i=1}^k a_ie_{n_i}\right|\bigg\|&\leq&\sum_{l=1}^{m+1}c_l\bigg\|\bigvee_{k=1}^m\left|\sum_{i=1}^k \varepsilon_l^ie_{n_i}\right|\bigg\|\leq 2C_{qg}^\vee\sum_{l=1}^{m+1}c_l\bigg\|\sum_{i=1}^me_{n_i}\bigg\|\\
&=&2C_{qg}^\vee \bigg\|\sum_{i=1}^me_{n_i}\bigg\|.
\end{eqnarray*}




\end{proof}

\begin{corollary}
Suppose that $(e_n)$ is a uniformly quasi-greedy basic sequence. Then for every $x\in E$, every greedy ordering $\pi$ of $x$ and every $m\in \N$,
$$\left|e^*_{\pi(m)}(x)\right|\bigg\|\bigvee_{k=1}^m\left|\sum_{i=1}^ke_{n_i}\right|\bigg\|\leq 4C_{qg}C_{qg}^\vee\|x\|,$$
where $\{n_1,\dots, n_m\}=\{\pi(1),\dots,\pi(m)\}$. In particular, for any distinct indices $n_1,\dots, n_m$, any real numbers $a_1,\dots,a_m$, and any permutation $\sigma$ of $\{1,\dots,m\}$
$$\min_{j\in\{1,\dots,m\}} |a_j|\bigg\|\bigvee_{k=1}^m\left|\sum_{i=1}^ke_{n_{\sigma(i)}}\right|\bigg\|\leq 4C_{qg}C_{qg}^\vee\bigg\|\sum_{i=1}^ma_ie_{n_i}\bigg\|.$$

\end{corollary}
\begin{proof}
The proof is the same as \cite[Theorem 10.2.12]{AK16} except that we replace  the first inequality by $\|\bigvee_{k=1}^m\left|\sum_{i=1}^ke_{n_i}\right|\|\leq 2C_{qg}^\vee \|\sum_{i=1}^m\varepsilon_ie_{\pi(i)}\|$, which is allowed by \Cref{almost unc}.

\end{proof}

We next introduce some notation to measure non-permutability. Let $(e_n)$ be a basic sequence in a Banach lattice $X$ and let $A=\{n_1,\dots,n_m\}$ be an ordered subset of $\N$. Define a continuous, sublinear, absolutely homogeneous map $P_A^\vee$ on $E$ by $P_A^\vee(x):=\bigvee_{i=1}^m\left|\sum_{k=1}^i e_{n_k}^*(x)e_{n_k}\right|$, where we take into account that $A$ is ordered via the order in which we sup up. Define   $\|P_A^\vee\|=\sup_{x\in S_E}\big\|P_A^\vee(x)\big\|$ and note that the following are equivalent:
\begin{itemize}
\item $(e_n)$ is bibasic;
\item For each $x\in E$ the sequence $(P_{\{1,\dots,m\}}^\vee(x))_m$ is norm bounded;
\item The sequence $(\|P_{\{1,\dots,m\}}^{\vee}\|)_m$ is norm bounded.
\end{itemize}
Hence, noting the trivial bound $\|P_A^\vee(x)\|\leq \textbf{k}|A|\|x\|$ where $\textbf{k}=\sup_n\|e_n\|\|e_n^*\|$, the maps $P_A^\vee$ can be used as a measure of how far $(e_n)$ is from being bibasic via the asymptotic growth of their norms. Taking into account more sets $A$, one can measure non-permutability. For uniformly quasi-greedy bases, we get the following bound.

\begin{corollary}\label{Similar asymptotics}
Suppose that $(e_n)$ is a uniformly quasi-greedy basic sequence. Then for every $x\in E$ and every ordered set $A=\{n_1,\dots,n_m\}\subseteq \text{supp}(x)$ we have
$$\|P_A^\vee (x)\|\leq 8C_{qg}^2C_{qg}^\vee\frac{\max\{|e_n^*(x)| : n\in A\}}{\min\{|e_n^*(x)| : n\in A\}}\|x\|.$$
\end{corollary}
\begin{proof}
Take $A=\{n_1,\dots,n_m\}\subseteq \text{supp}(x)$ and let  $$B=\{n\in\N : \alpha \leq \vert e_{n}^*(x)\vert\leq \beta\}$$ with $\alpha = \min_{A}\vert e_{n}^*(x)\vert$ and $\beta=\max_{A}\vert e_{n}^*(x)\vert$. We can extend the order on $A$ to the set $B$ by writing $B=\{n_1,...,n_m,n_{m+1},...,n_l\}$. Hence,
\begin{eqnarray}\label{one2}
\bigg\|\bigvee_{k=1}^m\left|\sum_{i=1}^k e_{n_i}^*(x)e_{n_i}\right|\bigg\|\leq \bigg\|\bigvee_{k=1}^l\left|\sum_{i=1}^k e_{n_i}^*(x)e_{n_i}\right|\bigg\|.
\end{eqnarray}
Applying  Corollary \ref{Convex max}, we obtain the estimate
\begin{eqnarray*}
\bigg\|\bigvee_{k=1}^l\left|\sum_{i=1}^k e_{n_i}^*(x)e_{n_i}\right|\bigg\|&\leq& 2C_{gq}^\vee \max_A\vert e_{n}^*(x)\vert \Vert \sum_{i\in B}\varepsilon_i e_{n_i}\Vert\\
&=& 2C_{qg}^\vee \frac{\max_A\vert e_{n}^*(x)\vert}{\min_A\vert e_{n}^*(x)\vert}\min_B \vert e_{n}^*(x)\vert  \Vert \sum_{i=1}^l\varepsilon_i e_{n_i}\Vert.
\end{eqnarray*}
Using \cite[Lemma 2.3 ]{BBG},  we see that
\begin{eqnarray}\label{two}
\min_B \vert e_{n}^*(x)\vert  \Vert \sum_{i\in B}\varepsilon_i e_{n_i}\Vert\leq 2C_{qg}\Vert \sum_{i=1}^le_{n_i}^*(x)e_{n_i}\Vert.
\end{eqnarray}
We next notice that for $\varepsilon >0$ small enough we have
$$\sum_{i=1}^l e_{n_i}^*(x)e_{n_i} = \sum_{\lbrace n_i: \vert e_{n_i}^*(x)\vert>\alpha-\varepsilon\rbrace}e_{n_i}^*(x)e_{n_i}- \sum_{\lbrace n_i: \vert e_{n_i}^*(x)\vert>\beta\rbrace}e_{n_i}^*(x)e_{n_i}.$$
Hence, $\Vert \sum_{i=1}^l e_{n_i}^*(x)e_{n_i}  \Vert \leq 2C_{qg}\Vert x\Vert.$ Combining this with \eqref{two} we obtain the estimate
\begin{eqnarray}\label{three}
\min_B \vert e_{n}^*(x)\vert  \Vert \sum_{i=1}^l\varepsilon_i e_{n_i}\Vert\leq 4C_{qg}^2\Vert x\Vert.
\end{eqnarray} 
Combining everything, we conclude that
$$\bigg\|\bigvee_{k=1}^m\left|\sum_{i=1}^k e_{n_i}^*(x)e_{n_i}\right|\bigg\|\leq 8C_{qg}^\vee C_{qg}^2\frac{\max_A\vert e_{n}^*(x)\vert}{\min_A\vert e_{n}^*(x)\vert} \Vert x\Vert,$$
as desired

\end{proof}

\begin{remark}
To measure conditionality of $(e_n)$ it is standard to use the growth of the sequence $\textbf{k}_m:=\sup_{|A|\leq m}\|P_A\|$. As mentioned, the sequence $\textbf{k}_m^\vee:=\sup_{|A|\leq m}\|P_A^\vee\|$ is a measure of non-permutability, where the sup is over all ordered subsets of $\N$ of cardinality at most $m$.  Clearly, $\textbf{k}_m\leq \textbf{k}_m^\vee$.  By applying \Cref{Similar asymptotics}  to $P_A(x)$ and using  the identity $P_A^\vee(P_A(x))=P_A^\vee(x)$ we see that, for certain $x\in E$, we can bound $\|P_A^\vee(x)\|$ in terms of $\|P_A(x)\|$. However, in general, $\|P_A^\vee(x)\|$ and $\|P_A(x)\|$ behave qualitatively different. Indeed, the Haar in $L_p[0,1]$, $p>1$, is a uniformly quasi-greedy bibasis, it is unconditional so $\textbf{k}_m=\mathcal{O}(1)$, but it is not permutable so $\textbf{k}_m^\vee \neq \mathcal{O}(1)$. 
\end{remark}
\begin{proposition}\label{asymp bound of constant}
 Let  $(e_n)$ be a uniformly quasi-greedy basic sequence.  Then $\textbf{k}_m^\vee=\mathcal{O}(\log_2(m)).$
\end{proposition}
\begin{proof}
Consider an integer $m\geq 2$ and find $p$ such that $2^p\leq m<2^{p+1}$. Let $x\in S_E$ and note that we have $|e_n^*(x)|\leq \textbf{K}$ for all $n\in \N$, where $\textbf{K}:=\sup_n \|e_n^*\|$. Construct a partition $(B_j)_{j=0}^p$ of $\mathbb{N}$ as in the proof of \cite[Theorem 10.2.14]{AK16}.
\medskip

Let $A=\{n_1,\dots,n_m\}$ be an ordered subset of $\N$. Then for each $1\leq l\leq m$ we have 
$$\left|\sum_{k=1}^l e_{n_k}^*(x)e_{n_k}\right|\leq \left|\sum_{k\in\{1,\dots, l\}\cap B_0} e_{n_k}^*(x)e_{n_k}\right|+\cdots+ \left|\sum_{k\in\{1,\dots, l\}\cap B_p} e_{n_k}^*(x)e_{n_k}\right|.$$
Hence,
$$P_A^\vee(x)\leq \bigvee_{l=1}^m\left|\sum_{k\in\{1,\dots, l\}\cap B_0} e_{n_k}^*(x)e_{n_k}\right|+\cdots+ \bigvee_{l=1}^m\left|\sum_{k\in\{1,\dots, l\}\cap B_p} e_{n_k}^*(x)e_{n_k}\right|.$$
For $j=1,\dots, p$, \Cref{Similar asymptotics} gives $$\bigg\|\bigvee_{l=1}^m\left|\sum_{k\in\{1,\dots, l\}\cap B_j} e_{n_k}^*(x)e_{n_k}\right|\bigg\|\leq 16 C_{qg}^2C_{qg}^\vee,$$ while the trivial estimate gives $$\bigg\|\bigvee_{l=1}^m\left|\sum_{k\in\{1,\dots, l\}\cap B_0} e_{n_k}^*(x)e_{n_k}\right|\bigg\|\leq m\textbf{K}2^{-p}c\leq 2c\textbf{K},$$
with $c=\sup_n \|e_n\|$. Combining the above estimates yields the bound $\|P_A^\vee(x)\|\leq 16pC_{qg}^2C_{qg}^\vee+2c\textbf{K}$, as required.
\end{proof}
\begin{remark}
One cannot hope for a better bound in \Cref{asymp bound of constant}, in general, as one can consider embeddings into $C[0,1]$ (the estimate $\textbf{k}_m=\mathcal{O}(\log_2(m))$ for quasi-greedy bases is sharp for general Banach spaces). One may hope for better estimates if one only considers bases, or if one works in nicer spaces, e.g., $L_p(\mu)$. In this direction, martingale theory gives that the bibasis constant of the standard ordering of the Haar in $L_p[0,1]$ is $q$, where $\frac{1}{p}+\frac{1}{q}=1$. Knowing this exact bound, it is natural to inquire what the value -- or at least the asymptotics -- of $\textbf{k}_m^\vee$, $C_{qg}^\vee$  and the bibasis constant are for other bases. In particular, which bases of $L_p[0,1]$ minimize these quantities? Although we do not have a complete answer to this question, in \Cref{Lp no perm} we will show that $\textbf{k}_m^\vee\neq \mathcal{O}(1)$ for any basis of $L_p[0,1]$.
\end{remark}



\section{The role of the ambient space}\label{s:destroy}
In \cite{TT19} and \Cref{Sub3} it is shown that many of the major results about Schauder and quasi-greedy bases generalize to bibases and uniformly quasi-greedy bases. In this section, we delve deeper into the interplay between coordinate systems and lattice structures and prove  results that have no direct analogues in the classical theories.
\subsection{Embedding bases into Banach lattices}\label{embed bad}
Given any basis $(e_i)_i$ of a Banach space $E$, we may always embed $E$ into a Banach lattice $X$ so that $(e_i)_i$ becomes bibasic, and hence has good order properties. In this subsection, we aim to do the following: 
\begin{itemize}
 \item Take a basis $(e_i)_i$ of a Banach space $E$ possessing certain good properties.
 \item Embed $E$ into a Banach lattice $X$ so that the ``uniform'' analogues of these properties fail. 
\end{itemize}
More specifically, we deal with:
\begin{enumerate}
 \item 
 The Lindenstrauss basis in $\ell_1$, or rather its version described in \cite{DM}; this is an example of a conditional quasi-greedy basis.
 \item
 The canonical basis in $\ell_p$, $1<p<\infty$.
\end{enumerate}

\subsubsection{A copy of the  Lindenstrauss basis that is neither bibasic nor uniformly quasi-greedy}\label{ss:destroy_Lind}
A fundamental example of a conditional quasi-greedy basis in $\ell_1$ is given in \cite{DM}. We begin by outlining their construction.
\medskip

For $n \in \N$ and $j \in \{1,2\}$, define $\phi_j(n) = 2n+j$. Denote by $(e_i)$ the canonical basis of $\ell_1$, and let $x_n = e_n - \big(e_{\phi_1(n)} + e_{\phi_2(n)}\big)/2$.
Denote by $F_n$ the span of $x_1, \ldots, x_n$ in $\ell_1^{2n+2}$.
By \cite{DM}, each $F_n$ is $C$-isomorphic to $\ell_1^n$, where $C$ is a uniform constant.
Denote by $x_1^{(n)}, \ldots, x_n^{(n)}$ the copies of $x_1, \ldots, x_n$ inside of $F_n$.
By \cite{DM}, the vectors $\big(x_i^{(n)}\big)_{1 \leq i \leq n, n \in \N}$ form a quasi-greedy basis inside of $\big( \sum_n F_n \big)_1 \sim \ell_1$. Now equip ${\mathrm{span}} \big[ \big(x_i^{(n)}\big)_{1 \leq i \leq n} \big]$ with the lattice structure of $\big( \sum_n \ell_1^{2n+2} \big)_1$, which contains $\big( \sum_n F_n \big)_1$ in the natural way.
\begin{proposition}\label{p:destroy_Lind}
 In the above notation, the sequence $\big(x_i^{(n)}\big)_{1 \leq i \leq n, n \in \N}$ is neither uniformly quasi-greedy nor bibasic.
\end{proposition}
\begin{proof}
For any $C > 1$, find $m \in \N$ with $m+1 > C/2$.
 For $n \geq 3 \cdot 2^m$, we construct $x \in {\mathrm{span}} \big[ \big(x_i^{(n)}\big)_{1 \leq i \leq n} \big]$ which witnesses the fact that the uniformly quasi-greedy and  bibasis constants of the sequence $\big(x_i^{(n)}\big)_{1 \leq i \leq n}$ in $\ell_1^{2n+2}$ are at least $C$.
 Since $n$ is fixed, we shall use $x_i$ instead of $x_i^{(n)}$.
 \medskip

 For $\overline{j} = ( j_1, \ldots, j_k ) \in \{1,2\}^k$ define $\phi_{\overline{j}}(t) = \phi_{j_k} \big( \ldots \phi_{j_2}(\phi_{j_1}(t)) \ldots \big)$; then
 $\phi_{\overline{j}}(t) = 2^k t + 2^{k-1} j_1 + \ldots + 2 j_{k-1} + j_k$.
 For convenience, let $\phi_\emptyset(t) = t$.
 Note that, if $\overline{j} = ( j_1, \ldots, j_k )$ and $\overline{i} = ( i_1, \ldots, i_\ell )$, then $\phi_{\overline{j}}(1) = \phi_{\overline{i}}(1)$ holds iff $\ell = k$ and $j_1=i_1,\dots, j_k=i_k$. Indeed, for $k = \ell$ this follows from the uniqueness of the binary decomposition. On the other hand, if $\ell > k$ then
 $$
 \phi_{\overline{j}}(1) \leq 2^k + 2 \big( 2^{k-1} + \ldots + 1 \big) < 2^k + \big( 2^k + 2^{k-1} + \ldots + 1 \big) \leq \phi_{\overline{i}}(1) .
 $$
 The case of $\ell < k$ is handled in a similar way.
 \medskip
 
 For $m \in \N$, let
 $$
 y_m = \sum_{k=0}^{m-1} 2^{-k} \sum_{j_1, \ldots, j_k \in \{1,2\}} x_{\phi_{j_k, \ldots, j_1}(1)},
 $$
 where, by convention, the term for $k=0$ corresponds to $x_{\phi_\emptyset(1)} = x_1$.
 A direct computation shows that
 \begin{equation}
     \label{eq:y-m1}
 y_m = e_1 - 2^{-m} \sum_{j_1, \ldots, j_m \in \{1,2\}} e_{\phi(j_m, \ldots, j_1)(1)} .
 \end{equation}
 By the definition of $\phi$,
 \begin{equation}
     \label{eq:y-m2}
 \sum_{j_1, \ldots, j_m \in \{1,2\}} e_{\phi(j_m, \ldots, j_1)(1)} = 
 \sum_{i \in I_m} e_i, \, {\textrm{  where   }} \, I_m = \{2^{m+1}-1 , \ldots , 2^{m+1}+2^m-2 \} ,
 \end{equation}
 hence $\|y_m\| = 2$.
\medskip

 Next, we show that, for any $N \in \N$, $\| \vee_{m=0}^{N-1} |y_m| \| = N+1$. To this end, observe that, in \eqref{eq:y-m2}, $|I_m| = 2^m$, and $I_m \cap I_k = \emptyset$ if $m \neq k$. Therefore,
 $\vee_{m=0}^{N-1} |y_m| = e_1 + \sum_{m=1}^N 2^{-m} \sum_{i \in I_m} e_i$, yielding the desired estimate.
\medskip

 Note that, if the Lindenstrauss basis is bibasic (uniformly quasi-greedy), then there exists ${\mathbf{B}} \in (0,\infty)$ so that $(x_i)_{i=1}^n$ is bibasic (resp.~uniformly quasi-greedy) with constant ${\mathbf{B}}$, no matter how large $n$ is.
Fix $N$, and pick $n \geq \phi_{2,\ldots,2}(1)$ ($2$ is repeated $N$ times).
If $m \leq N$, then $y_m$ is both a greedy sum and a partial sum of $y_N$, hence $\| \vee_{m=0}^{N-1} |y_m| \| \leq {\mathbf{B}} \|y_N\|$. This gives ${\mathbf{B}} \geq(N+1)/2$ (here ${\mathbf{B}}$ can stand for either the bibasic or the uniformly quasi-greedy constant of $(x_i)_{i=1}^n$), leading to the desired contradiction.
\end{proof}

\subsubsection{A copy of the canonical \texorpdfstring{$\ell_p$}{}-basis that is not bibasic nor uniformly quasi-greedy}\label{ss:destroy_lp}
Here, we shall denote by $(e_i)$ the canonical $\ell_p$-basis. We establish the following proposition, which partially resolves a question in \cite[Remark 7.6]{OTTT}.

\begin{proposition}\label{p:destroy_lp}
 For $1 < p < \infty$ the space $\ell_p$ contains a basis $(u_i)$, equivalent to $(e_i)$,  which is neither bibasic nor uniformly quasi-greedy.
\end{proposition}

\begin{proof}
Our reasoning is similar to \cite[Section 7.1]{OTTT} and relies on the investigation of the main triangular truncation carried out in \cite{KP}. 
\medskip

For each $n$, consider the $n \times n$ matrix $T^{(n)} = \big[ T^{(n)}_{ij} \big]$ with $T^{(n)}_{ij} = 1/(i-j)$ if $i \neq j$ and $T^{(n)}_{ij} = 0$ if $i=j$.
As noted in \cite[(1.7)]{KP} (where a ``mirror image'' of this matrix is considered), $\|T^{(n)}\| \leq K(p)$ with $K(p)$  a uniform constant and $T^{(n)}$  viewed as an operator on $\ell_p^n$.
Let $S^{(n)} = \alpha T^{(n)}$ where $\alpha = 1/(2 K(p))$.
When the value of $n$ is not in doubt, we shall write $S$ and $T$ instead of $S^{(n)}$ and $T^{(n)}$.
\medskip

Identify $\ell_p$ with $\big( \sum_n (\ell_p^n \oplus_p \ell_p^n) \big)_p$.
Denote by $(f_i)_{i=1}^n$ and $(g_i)_{i=1}^n$ the canonical bases of the first and the second copies of $\ell_p^n$, respectively.
Inside of $\ell_p^n \oplus_p \ell_p^n$ consider the basis consisting of elements $v_i, w_i$ ($1 \leq i \leq n$) with $v_i = f_i \oplus S^{(n)} g_i$ and $w_i = - S^{(n)} f_i \oplus g_i$. 
We claim that the basis $(v_i, w_i)$ is equivalent to $(e_i)_{i=1}^{2n}$, with the constant of equivalence independent of $n$.
Indeed, we can write $\sum_{i=1}^n (\alpha_i u_i + \beta_i v_i) = A \sum_i (\alpha_i f_i + \beta_i g_i)$, where
$$
A = \begin{pmatrix}  I & S  \\  -S & I  \end{pmatrix} = I + \begin{pmatrix}  0 & S  \\  -S & 0  \end{pmatrix} .
$$
Clearly, $\left\| \begin{pmatrix}  0 & S  \\  -S & 0  \end{pmatrix} \right\| \leq \frac12$. Hence, $\|A\| \leq \frac32$. A Neumann series expansion gives us the identity
$$
A^{-1} = I + \sum_{k=1}^\infty (-1)^k \begin{pmatrix}  0 & S  \\  -S & 0  \end{pmatrix}^k ,
$$
hence $\|A^{-1}\| \leq 2$.
\medskip

Concatenate the bases $(v_i, w_i)$ into $(u_i)$. As shown above, $(u_i)$ is equivalent to the $\ell_p$ basis. Now we show that $(v_i, w_i)$ is neither bibasic in $\ell_p^n \oplus_p \ell_p^n$ nor uniformly quasi-greedy there. To this end, consider $x = \sum_{i=1}^n v_i$. Clearly, $\|x\| \leq 3 n^{1/p}/2$. We shall show that
$$
\Big\| \bigvee_{k=1}^n \Big| \sum_{i=1}^k v_i \Big| \Big\| \succ n^{1/p} \log n .
$$
To estimate the left-hand side, we focus on the second copy of $\ell_p^n$. For $2 \leq j \leq n$, the $j$-th term of $\sum_{i=1}^j v_i$ (that is, the coefficient in front of $g_j$) equals
$\alpha \big( \frac1{j-1} + \ldots + 1 \big)$. Consequently,
$$
\bigvee_{k=1}^n \Big| \sum_{i=1}^k v_i \Big| \geq \alpha \sum_{j=2}^n \Big( \frac1{j-1} + \ldots + 1 \Big) g_j .
$$
The term in front of $g_j$ is $\succ \log n$ for $j \geq n/2$, hence the norm of the left-hand side $\succ n^{1/p} \log n$. This finishes the proof.
\end{proof}

In the specific case of $p=2$, one can do somewhat better.

\begin{proposition}\label{p:destroy_l2_isometrically}
 The space $\ell_2$ contains a basis $(u_i)$, isometrically equivalent to $(e_i)$,  which is neither bibasic nor uniformly quasi-greedy.
\end{proposition}
This is a consequence of the following lemma.

\begin{lemma}\label{l:unitaries}
 There is a constant $c$ such that for every $n$ there is an orthonormal basis $h_1, \ldots, h_n$ of $\ell_2^n$ so that
 $$
 \Big\| \bigvee_{k=1}^n \big|h_1 + \ldots + h_k\big| \Big\| \geq c \log n \, \sqrt{n} .
 $$
\end{lemma}

\begin{proof}[Sketch of a proof of \Cref{p:destroy_l2_isometrically}]
We closely follow the proof of \Cref{p:destroy_lp}. The only difference is that now (in the same notation as before) we take $v_i = 12 f_i/13 \oplus 5 h_i/13$ and $w_i = -5 h_i/13 \oplus 12 g_i/13$, with $h_1, \ldots, h_n$ coming from \Cref{l:unitaries}.
\end{proof}


\begin{proof}[Proof of \Cref{l:unitaries}]
Note that there exists a bijection between orthonormal bases $h_1, \ldots, h_n$ and $n \times n$ unitary matrices $U$, implemented by $U \mapsto (Ue_1, \ldots, Ue_n)$.
By the Russo-Dye Theorem (for the real version, see \cite{unitary-operators}), the unit ball of $B(\ell_2^n)$ is the closed convex hull of unitaries, hence it suffices to show that the operator
$$
\Phi : B(\ell_2^n) \to \ell_2^n(\ell_\infty^n) : T \mapsto \big( Te_1, Te_1 + Te_2, \ldots, Te_1 + \ldots + Te_n \big)
$$
has norm $\succ \log n \sqrt{n}$. Here, we define
$$
\big\|(\eta_1, \ldots, \eta_n) \big\|_{\ell_2^n(\ell_\infty^n)} = \big\| \vee_k |\eta_k| \big\|_{\ell_2^n},
$$
or in other words, for $\eta_k = \sum_{\ell=1}^n \eta_{k\ell} e_\ell$,
$$
\big\|(\eta_1, \ldots, \eta_n) \big\|_{\ell_2^n(\ell_\infty^n)} = \big\| \sum_\ell \big( \vee_k |\eta_{k\ell}| \big) e_\ell \big\|_{\ell_2^n} = \Big( \sum_\ell  \big( \vee_k |\eta_{k\ell}| \big)^2 \Big)^{1/2} .
$$
We shall actually establish the corresponding estimate for the norm of $\Phi^*$, which takes $\ell_2^n(\ell_1^n)$ to $B(\ell_2^n)^*$ (the space of trace class $n \times n$ matrices).
Here, we consider the ``trace duality'' on $B(\ell_2^n)$: for $n \times n$ matrices $A$ and $B$,
$\langle A, B \rangle = {\mathrm{tr}}(A B^*)$.
One can observe that $\Phi^*$ maps $\overline{\xi} = (\xi_1, \ldots, \xi_n)$ to the operator
$$
\Phi^* \big( \overline{\xi} \big) : e_k \mapsto \xi_k + \ldots + \xi_n \, \, (1 \leq k \leq n ) .
$$
In particular, taking $\xi_k = e_k$ ($1 \leq k \leq n$), we obtain
$$
\big\|\overline{\xi}\big\|_{\ell_2^n(\ell_1^n)} = \big\| \sum_k |e_k| \big\|_{\ell_2^n} = \sqrt{n} .
$$
On the other hand, $\Phi^*(\overline{\xi})$ is represented by the lower triangular matrix ${\mathbf{\tau}} = \big[ {\mathbf{\tau}}_{k\ell} \big]_{k,\ell=1}^n$, with ${\mathbf{\tau}}_{k\ell} = 1$ if $k \geq \ell$, ${\mathbf{\tau}}_{k\ell} =0$ otherwise. It remains to show that $\big\| {\mathbf{\tau}} \big\|_{B(\ell_2^n)^*} \succ n \log n$.
\medskip

To obtain this estimate, consider the $n \times n$ Toeplitz matrix $A = \big[ A_{k\ell} \big]$, with $A_{k\ell} = 1/(k-\ell)$ if $k \neq \ell$ and $A_{k\ell}=0$ if $k=\ell$.
By \cite{ACN} (or \cite{KP}), $\|A\|_{B(\ell_2^n)} \leq \pi$. 
By parallel duality,
$$
\|A\|_{B(\ell_2)} \big\| {\mathbf{\tau}} \big\|_{B(\ell_2^n)^*} \geq \langle A, {\mathbf{\tau}} \rangle = \sum_{k,\ell} A_{k\ell} {\mathbf{\tau}}_{k\ell} = \sum_{m=1}^n \sum_{j=1}^m \frac1j \sim n \log n ,
$$
which yields the desired estimate for $\big\| {\mathbf{\tau}} \big\|_{B(\ell_2^n)^*}$.
\end{proof}

\subsubsection{$\ell_1$ is the only basis that is  absolute in every Banach lattice}\label{ss:uncondiitonal bases}

Here we give an alternative proof of one of the results of \cite{OTTT}.

\begin{proposition}\label{p:destroy_arbitrary_uncinditional}
Suppose that $(e_i)$ is a semi-normalized  basis in a Banach space $E$, not equivalent to the $\ell_1$ basis.
Then there exists a Banach lattice $X$ containing a basic sequence $(x_i)$ equivalent to $(e_i)$ so that $(x_i)$ is not an absolute basic sequence.
\end{proposition}
\begin{remark}
 We do not know whether this proposition can be strengthened in one of the following ways:
 \begin{itemize}
  \item Can we replace ``absolute'' by a weaker condition such as ``permutable bibasic?''
  \item Can $(x_i)$ be a basis and not just a basic sequence? (here we are implicitly assuming that $(x_i)$ is unconditional  so that $E=[x_i]$ admits lattice structures).
 \end{itemize}
\end{remark}

\begin{proof}
Any absolute basis is unconditional, so if $(e_i)$ is conditional an embedding of $E$ into a suitable $C(K)$ will generate the desired $(x_i)$. We henceforth assume that $(e_i)$ is an unconditional basis of $E$. Due to the uniqueness of  unconditional bases in $\ell_1$ \cite[Theorem 2.b.10]{LT1}, $E$ is not isomorphic to $\ell_1$.
By renorming, we can assume that $(e_i)$ is $1$-unconditional.
We shall view $E$ as a Banach lattice with the order determined by $(e_i)$.
\medskip

Fix $p \in [1,\infty)$. By \cite[Theorem 2.3]{Wickstead07}, for any $N \in \N$ and $C>0$ there exists a finite rank contraction $T : \spn[e_i : i > N] \to \ell_p$ of regular norm greater than $C$.  That is,
$\||T|\|>C $.
Note that $\ell_p$ is positively contractively complemented in its second dual, and hence the regular norm is equal to the order bounded norm.
Approximating, we can find a norm one $x = \sum_i \alpha_i e_i$ (finite sum, with $\alpha_i \geq 0$) so that
$$
\Big\| \vee_{\varepsilon_i = \pm 1} \big| \sum_i \varepsilon_i \alpha_i T e_i \big| \Big\| > C .
$$
Note that
$$
\vee_{\varepsilon_i = \pm 1} \big| \sum_i \varepsilon_i \alpha_i T e_i\big| = \sum_i \big| \alpha_i T e_i \big| ,
$$
so we have, in fact,
$$
\Big\| \sum_i \big| \alpha_i T e_i \big| \Big\| > C .
$$
This allows us to find $1 = N_0 < N_1 < N_2 < \ldots$, contractions $T_j : \spn[e_i : N_{j-1} \leq i < N_j] \to \ell_p$, and $\alpha_i \geq 0$ so that
$$
\big\|\sum_{i=N_{j-1}}^{N_j-1} \alpha_i e_i\big\| = 1 , \, \, \,
\Big\| \sum_i \alpha_i \big| T_j e_i \big| \Big\| > j .
$$
In the space $X = E \oplus_\infty c_0(\ell_p)$, let $x_i = e_i \oplus 0 \oplus \ldots \oplus 0 \oplus T_j e_i \oplus 0 \oplus \ldots$ (the term $T_j e_i$ is located in the $j$-th copy of $\ell_p$).
The sequence $(x_i)$ is equivalent to $(e_i)$ as each $T_j$ is a contraction, but $(x_i)$ is not absolute in $X$.
\end{proof}

\subsection{Blocking bases}\label{Section5}
In this subsection, we show that the absolute and bibasis properties behave very well under blocking, whereas the greedy property does not.
\subsubsection{Blocking to gain the absolute property}\label{section 5.1}
In \Cref{embed bad} we have given several methods to construct basic sequences with poor lattice properties. We now present some techniques to produce good basic sequences (or more specifically FDDs, from which one can extract good basic sequences). The main insight is that most bases are built in disjoint blocks, so if we lump the blocks together, we can get better lattice behavior.

\begin{proposition}\label{Blocking to get a-FDD}
Let $X$ be a Banach lattice and $(C_n)$ an absolute FDD of $[C_n]\subseteq X$. Then every unconditional and shrinking FDD $(B_n)$ with $[B_n]\subseteq [C_n]$ can be blocked to be absolute.
\end{proposition}
\begin{proof}
Let $(B_n)$ be an unconditional and shrinking FDD with $[B_n]\subseteq [C_n]$.  Let $\varepsilon_i=\frac{1}{2^i}$ and apply \cite[Proposition 1.g.4]{LT1} to get a blocking $(B_i')$ of $(B_n)$ and a blocking $(C_i')$ of $(C_n)$ such that for every $x\in B_i'$ there is a $y\in C_{i-1}'\oplus C_i'$ with $\|x-y\|\leq \varepsilon_i\|x\|$.
\medskip

Take normalized $x_i\in B_i'\setminus \{0\}$; it suffices to show that $(x_i)$ is absolute. For this, we let $x\in [x_i]$, write $x=\sum_{i=1}^\infty a_ix_i$, and find $y_i\in C_{i-1}'\oplus C_i'$ with $\|a_ix_i-y_i\|\leq \varepsilon_i |a_i|$. By unconditionality,
$x^1=\sum_{i\ odd}a_ix_i$ and $x^2=\sum_{i\ even} a_ix_i$ exist.  Moreover, we have 

$$\|\sum_{i\ odd}|a_ix_i|\|\leq\|\sum_{i \ odd}|a_ix_i-y_i|\|+\|\sum_{i\ odd}|y_i|\|\leq \sum_{i \ odd}\varepsilon_i |a_i|+\|\sum_{i\ odd}|y_i|\|<\infty.$$ 
The reason the above is finite is because the sequence $(x_i)$ is normalized, $a_i\to 0$, and since we are only summing over odd terms the $y_i$ do not overlap. This means that $\|\sum_{i\ odd}|y_i|\|\leq A\|\sum_{i\ odd}y_i\|$ and $\|\sum_{i\ odd}y_i\|<\infty$ since $\sum_{i\ odd}(a_ix_i-y_i)$ exists. Similarly, $\|\sum_{i\ even} |a_ix_i|\|<\infty$, so $\|\sum_{i=1}^\infty |a_ix_i|\|<\infty$ and hence $(x_i)$ is absolute.
\end{proof}
\begin{remark}
By \cite{Johnson-Zippin72}, every FDD of $\ell_p$, $p>1$, can be blocked to be unconditional.
As the canonical basis of $\ell_p$ is absolute, \Cref{Blocking to get a-FDD} shows that every FDD of $\ell_p$, $p>1$, can be blocked to be an absolute FDD. By contrast, the ``dual summing basis'' in $\ell_1$ cannot even be blocked as a bi-FDD. 
\end{remark}
Next, we present two examples to illustrate the sharpness of the assumptions in \Cref{Blocking to get a-FDD}.
\begin{example}
Based on the reasoning from \cite[Example 2.13]{JO81}, we show that $X=\ell_p\oplus \ell_q$ ($1 < q < p < \infty$) has a shrinking FDD which cannot be blocked to be either unconditional or a bi-FDD.
Denote by $(\delta_i)$ and $(e_i)$ the canonical bases of $\ell_q$ and  $\ell_p$, respectively. Let $E_1=[0\oplus \delta_1]$ and for $n\geq 2$ let $E_n=[e_{n-1}\oplus \delta_{n-1},0\oplus \delta_n]$.
For a blocking $F_n = [E_i : k(n) < i \leq k(n+1)]$, take $f_1 = 0 \oplus e_{k(2)} \in F_1$, and, for $2 \leq n \leq m$, $f_n = e_{k(n)} \oplus (\delta_{k(n)} + \delta_{k(n+1)}) \in F_n$.
It is easy to see that $\|\sum_{n=1}^m (-1)^n f_n\| \sim m^{1/p}$, while  $\|\sum_{n=1}^m f_n\| \sim m^{1/q}$, and also, $\|\vee_{k=1}^m |\sum_{n=1}^k (-1)^n f_n| \| \sim m^{1/q}$.
\end{example}


\begin{example}\label{bad weakly null}
A \emph{Haar system} is a family of functions $f_{i,n}$ ($n \geq 0, 1 \leq i \leq 2^n$) so that there exist sets $A_{i.n} \subset (0,1)$ of measure $2^{-n}$ so that $A_{i,n} = A_{2i-1,n+1} \cup A_{2i,n+1}$, $A_{i,n} \cap A_{j,n} = \emptyset$ if $i \neq j$, $f_{i,n} = 1$ on $A_{2i-1,n+1}$, and $f_{i,n} = -1$ on $A_{2i,n+1}$.
\cite{JMS07} gives an example of a weakly null normalized sequence $(g_k) \subset L_1(0,1)$ so that, for every subsequence $g_k'$, and every sequence of positive numbers $\varepsilon_{i,n}$, there exists a Haar sequence $h_{i,n}$ and a block sequence $g_{i,n}''$ of $g_k'$ so that $\|g_{i,n}'' - h_{i,n}\| < \varepsilon_{i,n}$.
By \cite[Theorem 4.6]{GKP}, ``branches'' of any Haar system fail to be bibasic. Hence, no subsequence of $(g_k)$ is bibasic. In particular, although every norm convergent sequence has a uniformly convergent subsequence, it is not true that every basic sequence admits a bibasic subsequence. Actually, much stronger results can be shown if one allows the ambient Banach lattice to be non-classical -- in \cite[Theorem 7.5]{OTTT} it is shown that there are subspaces with no bibasic sequences at all.
\end{example}

\subsubsection{Blocking to lose the greedy property}\label{lose}
In contrast to the situation in \Cref{section 5.1}, we now show that greediness is very unstable under blockings. 
For this purpose, it is instructive to look at an example. Recall that we are starting with a greedy basis $(e_i)$, which we want to block into an FDD which fails the democracy condition. To begin, let us take $(e_i)$ to be the standard basis of the Lorentz space $\ell_{p,q}$. One can find a sequence of mutually disjoint elements which is equivalent to $\ell_q$. Thus, one can block the canonical basis for $\ell_{p,q}$ into an FDD $(E_i)$ so that 
\begin{itemize}
    \item For $i$ odd, the space $E_i$ is $1$-dimensional, and spanned by a canonical basis element.
    \item For $i$ even, $E_i$ contains a unit vector $u_{i/2}$ such that $u_1,u_2,\dots$ span a copy of $\ell_q$.
\end{itemize}
The resulting FDD is not democratic. Indeed, the fundamental function evaluated from $E_i$'s with $i$ odd gives $\sim n^\frac{1}{p}$; on the other hand, if we look at $u_j\in E_{2j}$, we obtain the fundamental function $\sim n^\frac{1}{q.}$
\medskip

In a similar fashion, one can use the fact that the space $L_p$ with $p\in (1,\infty)\setminus\{2\}$ contains a Hilbert space to show that the Haar basis, although greedy, can be blocked into a non-greedy FDD. In general, we have the following.
\begin{proposition}\label{Rearrange greedy}
For a subsymmetric basis $(e_i)$ of a Banach space $X$, the following are equivalent.
\begin{enumerate}
    \item $(e_i)$ is equivalent to the canonical $\ell_p$ ($1\leq p<\infty)$ or $c_0$ basis;
    \item Every $($not necessarily consecutive$)$ blocking of a permutation of $(e_i)$ produces a greedy FDD.
\end{enumerate}
\end{proposition}
\begin{proof}
The implication (i)$\Rightarrow$(ii) is immediate, so we focus on (ii)$\Rightarrow$(i).
\medskip

We assume that every blocking of a permutation of $(e_i)$ is a greedy FDD.
By renorming, we can and do assume that $(e_i)$ is normalized, and $1$-unconditional.  Thus,  any FDD we produce from $(e_i)$ will be $1$-unconditional as well. We let $\phi(n)=\Big\|\sum_{i=1}^ne_i\Big\|$ and claim that there exists a constant $C$ so that, whenever $(x_k)_{k=1}^n$ is a disjoint collection of norm one vectors, then 
\begin{equation}\label{Dem estimate}
C^{-1}\phi(n)\leq \Big\|\sum_{k=1}^nx_k\Big\|\leq C\phi(n).\end{equation}
Suppose for the sake of contradiction that \eqref{Dem estimate} fails. Then for any $M\in \mathbb{N}$ and $C>1$ there exists $n\in \mathbb{N}$ and $x_1,\dots,x_n$ disjointly supported on $[M,\infty)$ for which \eqref{Dem estimate} fails. Then, concatenating, we can find a sequence $(N_k)$ so that $N_0=1$ and $N_k>2N_{k-1}$ for any $k$, together with norm one vectors $(x_{ik})$ ($k\in \mathbb{N}$, $1\leq i\leq n_k$), so that $(x_{ik})_{i=1}^{n_k}$ is disjointly supported on $[2N_{k-1},N_k-1]$ and

$$\Big\|\sum_{i=1}^{n_k}x_{ik}\Big\|\not\in [k^{-1}\phi(n_k),k\phi(n_k)].$$
We then consider an FDD so that
\begin{enumerate}
    \item Some blocks contain the vectors $x_{ik}$ described above;
    \item If $N_k\leq j<2N_k$ for some $k$, then span$[e_j]$ is a one-dimensional member of this FDD.
\end{enumerate}
Recall that, as $(e_i)$ is itself greedy, there exists a constant $c$ so that
$$c^{-1}\phi(n)\leq \Big\|\sum_{i\in A}e_i\Big\|\leq c\phi(n)$$
whenever $A\subseteq \mathbb{N}$ has cardinality $n$. Therefore, the FDD constructed above fails the democracy condition, so cannot be greedy. This is the desired contradiction.
\medskip

Reasoning as in the proof of \cite[Theorem 2.a.9]{LT1}, we show that there exists $\gamma\in [0,1]$ so that $C^{-2}n^\gamma\leq \phi(n)\leq C^2n^\gamma$ for every $n$, where $C$ is as before. If $\gamma=0$, then clearly $(e_i)$ is equivalent to the $c_0$ basis. We show that, if $\gamma>0$, then $(e_i)$ is equivalent to the $\ell_p$ basis, for $p=\frac{1}{\gamma}.$
\medskip

Consider a finitely supported norm one vector $a=\sum_{i\in S}\alpha_ie_i.$ We have to show that $\sum_i|\alpha_i|^p\sim 1.$ Due to unconditionality, we can and do assume that $\alpha_i>0$ for all $i\in S.$ Clearly, $\sup_i\alpha_i\leq 1$. For $j\in \mathbb{N},$ let $S_j=\{i\in S : 2^{-j/p}<\alpha_i\leq 2^{(1-j)/p}\}$. Let $J$ be the largest index $j$ for which $S_j$ is non-empty. Let $y=\sum_j2^{-j/p}\one_{S_j}$. Then $2^{-1/p}\leq \|y\|\leq 1$. It therefore suffices to show that $\sum_j 2^{-j}|S_j|\sim 1.$
\medskip

Fix $M > \max S$, divisible by $2^J$. For $0\leq k\leq M-1$, let $y_k=\sum_j2^{-j/p}\one_{S_j+kM}$, and 
$$y'=\sum_{k=0}^{M-1}2^{-j/p}\one_{S_j'}, \ \text{where} \ S_j'=\cup_{k=0}^{M-1}(S_j+(k-1)M).$$
Note that $\|y_k\|\sim 1$ by subsymmetry, and hence $\|y'\|\sim\phi(M)\sim M^{1/p}$. 
\medskip

For every $j\leq J,$ $2^{-j}M|S_j| = L_j$ is an integer.  $S_j'$ can be written as a union of $L_j$ disjoint sets $F_{j1},\dots, F_{jL_j}$, of cardinality $2^j$. Write
$$y'=\sum_j\sum_{\ell=1}^{L_j}2^{-j/p}\one_{F_{j\ell}}.$$
Each of the vectors $2^{-j/p}\one_{F_{j\ell}}$ has norm $\sim 1$, hence $M\sim \|y'\|^p\sim \sum_jL_j\sim \sum_j2^{-j}M|S_j|,$ which yields the desired result; namely, $\sum_j2^{-j}|S_j|\sim 1$.
\end{proof}

\begin{remark}
We do not know to what extent the subsymmetry assumption in \Cref{Rearrange greedy} is redundant. It is not required to establish that $\phi(n)\sim n^\frac{1}{p}$. If $(e_i)$ is not subsymmetric, our proof (with minor modifications) shows that any spreading basis $(e_i')$ arising from a subsequence of $(e_i)$ is equivalent to the $\ell_p$ basis. In fact, in the terminology of \cite{2}, $X$ must be strongly asymptotic $\ell_p$. In general, being strongly asymptotic $c_0$ does not imply being greedy, as the example of the Tsirelson space shows (see \cite[Remark 5.8]{3}).
\end{remark}
\begin{example}
It is shown in \cite{NoGreedy} that certain spaces of the form $(\oplus_{n=1}^{\infty} \ell_p^n)_{\ell_q}$ -- for appropriate $p,q$ -- do not have greedy bases. However, one can block the canonical basis in the evident way to get greedy FDD for these spaces. There are also more exotic spaces which have greedy FDD but no greedy bases; for example, \cite{Szarek} produces a space with an $\ell_2$-FDD but no basis.  On the other hand, there are several spaces for which it is unclear whether a greedy FDD can be produced; for example, $\ell_p\oplus \ell_q$, $\ell_q(\ell_p)$ (for appropriate $p,q$) and the Schatten classes. For  $\ell_p\oplus \ell_q$ it was shown in \cite{EW76,W78} that there is a unique unconditional basis up to permutation -- which is, of course, not greedy when $p\neq q$. See \cite{Schechtman} for a proof that there are no greedy bases for matrix spaces with mixed $\ell_p$ and $\ell_q$ norms. For the same reason as with $(\oplus_{n=1}^{\infty} \ell_p^n)_{\ell_q}$, it is clear that $\ell_q(\ell_p)$ has a greedy Schauder decomposition.
\end{example}

\begin{remark}\label{r:tensoring of bases}
 Results similar to \Cref{Rearrange greedy} (providing criteria for ``tensor-stable'' bases to be equivalent to the canonical basis of either $c_0$ or $\ell_p$) can be found in \cite{unit-vector-bases}.
\end{remark}

\subsection{Absolute sequences characterize AM-spaces}\label{ss:bad_basic} In AM-spaces (that is, closed sublattices of $C(K)$-spaces), every basic sequence is bibasic and every unconditional sequence is absolute. In \cite{TT19} it was asked whether the converse holds: If every basic (resp.~unconditional) sequence in $X$ is bibasic (resp.~absolute), must $X$ be lattice isomorphic to an AM-space? Here we make some progress on these conjectures.
\medskip

Recall that a Banach lattice $X$ has an \emph{upper $p$-estimate} if there exists a constant $C$ so that the inequality $\| \sum_i x_i\| \leq C \big( \sum_i \|x_i\|^p \big)^{1/p}$ holds for any disjoint $x_1, \ldots, x_n$. Clearly, an upper $p$-estimate implies upper $p'$-estimates for $p' < p$. Denote by $s(X)$ the supremum of all $p$'s for which $X$ has an upper $p$-estimate. By reversing inequalities, one obtains the definition of a lower $p$-estimate; denote by $\sigma(X)$ the infimum of all $p$'s for which $X$ has a lower $p$-estimate. We refer the reader to \cite{garcia2025banach,garcia2025banach2} for a comprehensive study of upper $p$-estimates.
Note that, if $X$ is finite dimensional, then $s(X) = \infty$ and $\sigma(X) = 1$. Otherwise, $1 \leq s(X) \leq \sigma(X) \leq \infty$. Furthermore, if $\sigma(X) < \infty$, then $X$ is order continuous. 

\begin{proposition}\label{p:uncond_not_bibasic}
 Suppose $X$ is an infinite dimensional Banach lattice, and $\{s(X), \sigma(X)\} \cap (1,\infty) \neq \emptyset$. 
 Then $X$ contains an unconditional sequence which is not bibasic.
\end{proposition}

Note that the hypothesis of the proposition fails if one of the three holds: (i) $s(X) = \sigma(X) = \infty$, (ii) $s(X) = \sigma(X) = 1$, or (iii) $s(X) = 1$, $\sigma(X) = \infty$.
Note that $s(x) = \sigma(X) = \infty$ does not imply that $X$ is an AM-space.
Similarly, $s(x) = \sigma(X) = 1$ does not imply that $X$ is an AL-space.
\medskip

For the proof, we need the following proposition, which strengthens Krivine's Theorem for lattices. 

\begin{proposition}\label{p:krivine}
    For a Banach lattice $X$, $c > 1$, and $n \in \N$, there exist disjoint normalized $x_i \in X_+$ $(1 \leq i \leq n)$ so that $(x_1, \ldots, x_n)$ is $c$-equivalent to the canonical basis of $\ell_{s(X)}^n$ (of $\ell_{\sigma(X)}^n$), and an infinite dimensional Banach lattice $Y \subset X$, disjoint from $x_1, \ldots, x_n$, so that $s(Y) = s(X)$ (resp.~$\sigma(Y) = \sigma(X)$).
\end{proposition}

Here, we say that the bases $(e_i)_{i=1}^n$ and $(f_i)_{i=1}^n$ are \emph{$c$-equivalent} if, for any sequence of scalars $(\alpha_i)_{i=1}^n$,
$$
c^{-1} \big\| \sum_i \alpha_i e_i \big\| \leq \big\| \sum_i \alpha_i f_i \big\| \leq c \big\| \sum_i \alpha_i e_i \big\| .
$$

We postpone the proof of this proposition and instead show how it implies \Cref{p:uncond_not_bibasic}.
\begin{proof}[Proof of \Cref{p:uncond_not_bibasic}]
We use \Cref{p:krivine} to find disjoint unit vectors $\big(e_{in}\big)_{n \in \N, 1 \leq i \leq 2n}$ so that, for every $n$, $\big(e_{in}\big)_{1 \leq i \leq 2n}$ is $2$-equivalent to the $\ell_r^{2n}$ basis, where $r \in \{s(X),\sigma(X)\} \cap (1,\infty)$.
By the construction from \Cref{p:destroy_lp}, $\spn [ e_{in} : 1 \leq i \leq 2n]$ has a basis $\big( u_{in} \big)_{i=1}^{2n}$ which is $c$-equivalent to the $\ell_r$ basis on the Banach space level, but has bibasis constant $\sim \log n$. Concatenating the $u_{in}$'s, one obtains an unconditional sequence in $X$ which is not bibasic.
\end{proof}

In the proof of \Cref{p:krivine}, we rely on some ideas from  \cite{Abramovich78} and \cite{AbLoz}. We also need two lemmas, the first of which is fairly straightforward.

\begin{lemma}\label{separability}
    Any Banach lattice $X$ contains a separable Banach lattice $X'$ so that $s(X') = s(X)$ and $\sigma(X') = \sigma(X)$.
\end{lemma}

Henceforth, we assume that $X$ is separable. We find a weak unit $u \in X_+$, and let $X_u$ be the corresponding principal ideal, with the norm $\|x\|_\infty = \inf \{ \lambda > 0 : |x| \leq \lambda u\}$.
This ideal can be identified with $C(K)$, for some compact Hausdorff space $K$, with $u$ corresponding to $1$. Adjusting $u$, we can assume that $\| \cdot \| \leq \| \cdot \|_\infty$.
As $X_u$ is $\| \cdot \|$-dense in $X$, we have $s(X) = s(X_u, \| \cdot \|)$ and $\sigma(X) = \sigma(X_u, \| \cdot \|)$.
\medskip

Suppose now that $\Omega$ is an open subset of $K$. We shall denote by $X^\Omega$ the set of all $x \in X_u$ which vanish outside of $\Omega$.

\begin{lemma}\label{good-point}
    There exists $t \in K$ so that, for every neighborhood $\Omega \ni t$, $s(X^\Omega) = s(X)$ $(\sigma(X^\Omega) = \sigma(X))$.
\end{lemma}

The spaces listed in this lemma are equipped with the norm $\| \cdot \|$. 

\begin{proof}
We deal with $s( \cdot )$, as $\sigma( \cdot )$ is handled similarly.
    Suppose, for the sake of contradiction, that any $t \in K$ has a neighborhood $\Omega_t$ so that $s(X^{\Omega_t}) > s(X)$. By the compactness of $K$, we can find $t_1, \ldots, t_n \in K$ so that $K = \cup_{i=1}^n \Omega_{t_i}$. We shall achieve a contradiction by showing that $s(X) = s:= \min_i s(X^{\Omega_{t_i}})$.
\medskip

    We shall show that, for any $r < s$, there exists a constant $C$ so that the inequality
    \begin{equation}
    \big\| \sum_{j=1}^m y_j \big\| \leq C \big( \sum_j \|y_j\|^r \big)^{1/r} 
    \label{upper est}
    \end{equation}
    holds for any disjoint $y_1, \ldots, y_m \in C(K)_+$.
\medskip

    Indeed, for any $i \in \{1, \ldots,n\}$ there exists $C_i$ so that, for any disjoint $z_1, \ldots, z_k \in C(K)_+$, vanishing outside of $\Omega_{t_i}$, we have
    $$
    \big\| \sum_{j=1}^k z_j \big\| \leq C_i \big( \sum_j \|z_j\|^r \big)^{1/r} .
    $$
 Find $x_1, \ldots, x_n \in C(K)_+$ so that $\sum_i x_i = 1$, and $x_i$ vanishes outside of $\Omega_{t_i}$. Let $y_{ij} = x_i \wedge y_j$. Then
    \begin{align*}
    \big\| \sum_{j=1}^m y_j \big\| 
    &
    \leq \sum_{i=1}^n \big\| \sum_{j=1}^m y_{ij} \big\| \leq \sum_i C_i \big( \sum_j \|y_{ij}\|^r \big)^{1/r}  
    \\
    &
    \leq \max_i C_i n^{1-1/r} \big( \sum_j \sum_i \|y_{ij}\|^r \big)^{1/r} \leq \max_i C_i n \big( \sum_j \|y_j\|^r \big)^{1/r}  ,
    \end{align*}
    the last inequality being due to the fact that
    $$
    \|y_j\| \geq \max_i \|y_{ij}\| \geq n^{-1/r} \big( \sum_i \|y_{ij}\|^r \big)^{1/r} .
    $$
    This establishes \eqref{upper est} with $C = n \cdot \max_i C_i$.
\medskip

    The case of $\sigma( \cdot )$ is handled similarly, except that one has to use the inequality $\big\| \sum_{j=1}^m y_j \big\| 
    \geq \vee_{i=1}^n \big\| \sum_{j=1}^m y_{ij} \big\|$ instead of $\big\| \sum_{j=1}^m y_j \big\| 
    \leq \sum_{i=1}^n \big\| \sum_{j=1}^m y_{ij} \big\|$.
\end{proof}

\begin{proof}[Proof of \Cref{p:krivine}]
Fix $\vp > 0$ and $c' \in (1,c)$ so that
$$
c > c' + n \vp \, {\textrm{  and  }} \, \frac1c < \frac1{c'} - n \vp .
$$
By Krivine's Theorem \cite{Schep}, $X$ contains disjoint normalized vectors $y_1, \ldots, y_{n+1}$, $c'$-equivalent to the $\ell_r$-basis, where $r$ is either $s(X)$ or $\sigma(X)$, depending on the case.
Perturbing these vectors slightly, we can assume that they belong to $X_u$ (in the notation of \Cref{good-point}); the latter space is identified with $C(K)$, again as in \Cref{good-point}.
\medskip

Let $t_0 \in K$ be the special point whose existence is guaranteed by \Cref{good-point}. At most one of our disjoint vectors $y_1, \ldots, y_{n+1}$ does not vanish at $t_0$; up to relabeling, we can assume that $y_1, \ldots, y_n$ do vanish there.
Find vectors $y_i' \in C(K)_+$ ($1 \leq i \leq n$) so that, for every $i$, $y_i'$ vanishes on some neighborhood of $t_0$, $y_i' \leq y_i$, and $\|y_i - y_i'\|_\infty < \vp$. Then $\|y_i - y_i'\| \leq \|y_i - y_i'\|_\infty < \vp$. 
\medskip

We claim that the vectors $x_i = y_i'/\|y_i'\|$ are disjoint, and $c$-equivalent to the $\ell_r$ basis. Indeed, for any $(\alpha_i)_{i=1}^n$,
\begin{align*}
\big\| \sum_i \alpha_i x_i \big\| 
&
\geq \big\| \sum_i \alpha_i y_i' \big\| \geq  \big\| \sum_i \alpha_i y_i \big\| - \sum_i |\alpha_i| \|y_i - y_i'\|
\\
&
\geq \frac1{c'} \big( \sum_i |\alpha_i|^r \big)^{1/r} - \vp \sum_i |\alpha_i| \geq \Big( \frac1{c'} - n \vp \Big) \big( \sum_i |\alpha_i|^r \big)^{1/r} .
\end{align*}
On the other hand, for any $i$,
$$
\big\| x_i - y_i' \big\| = \|y_i'\| \Big( \frac1{\|y_i\|} - 1 \Big) = 1 - \|y_i'\| < \vp ,
$$
hence
$$
\big\| \sum_i \alpha_i x_i \big\| \leq \big\| \sum_i \alpha_i y_i \big\| + \vp \sum_i |\alpha_i| \leq \big(c' + n \vp \big) \big( \sum_i |\alpha_i|^r \big)^{1/r} .
$$
Moreover, there exists an open neighborhood $\Omega \ni t_0$ disjoint from the $x_i$'s; then $X^\Omega$ has the required upper or lower estimate.
\end{proof}

In a manner similar to \Cref{p:uncond_not_bibasic} (but using \Cref{p:destroy_Lind} instead of \Cref{p:destroy_lp}), we establish the following.

\begin{proposition}\label{p:cond_not_bibasic}
 Suppose an infinite dimensional Banach lattice $X$ satisfies $s(X) = 1$.
 Then $X$ has a (conditional) basic sequence which is not bibasic.
\end{proposition}

Taken together, \Cref{p:uncond_not_bibasic} and \Cref{p:cond_not_bibasic} immediately imply the following.

\begin{corollary}\label{c:general basis destruction}
 Suppose an infinite dimensional Banach lattice $X$ has $s(X)< \infty$.
 Then $X$ contains a basic sequence which is not bibasic.
\end{corollary}
By \Cref{c:general basis destruction}, if every basic sequence in $X$ is bibasic then $s(X)=\infty.$ However, as mentioned above, $s(X)=\infty$ does not imply that $X$ is lattice isomorphic to an AM-space. We now show that if we instead  assume that every unconditional sequence in $X$ is absolute, we can reach this stronger conclusion.
\begin{theorem}\label{p:uncond_not_absolute}
If every unconditional basic sequence in a Banach lattice $X$ is absolute, then $X$ is lattice isomorphic to an AM-space.
\end{theorem}

For future use, we recall \cite[Theorem 2.1.12]{MN}: $X$ is lattice isomorphic to an AM-space iff there exists $C \geq 1$ so that the inequality $\| \vee_j x_j \| \leq C \vee_j \|x_j\|$ holds for any $x_1, \ldots, x_n \in X_+$ (actually it suffices to verify the preceding inequality for disjoint $n$-tuples only).
Consequently, $X$ is lattice isomorphic to an AM-space iff the same is true for any separable sublattice of $X$.
\medskip

In the proof, we re-use the notation and facts introduced earlier in this section. We begin by establishing a version of \Cref{p:krivine}.

\begin{lemma}\label{l:not AM}
Suppose a separable Banach lattice $X$ is not lattice isomorphic to an AM-space and $u$ is a weak unit in $X$.
There exists $t \in K$ so that, for every neighborhood of $\Omega \ni t$, the $\| \cdot \|$-completion of $X^\Omega$ is not lattice isomorphic to an AM-space.
\end{lemma}

\begin{proof}[Sketch of a proof]
Suppose, for the sake of contradiction, that every $t \in K$ possesses an open neighborhood $\Omega_t \ni t$ so that the completion of $X^{\Omega_t}$ is lattice isomorphic to an AM-space. Use compactness to find $t_1, \ldots, t_n \in K$ so that $\cup_{i=1}^n \Omega_{t_i} = K$. For every $i$ there exists a constant $C_i$ so that, whenever $z_1, \ldots, z_m \in X^{\Omega_{t_i}}$, we have $\|\vee_j |z_j|\| \leq C_i \vee_j \|z_j\|$.
\medskip

To achieve a contradiction, we show that there exists a constant $C$ so that, for any $y_1, \ldots, y_m \in C(K)_+$ (we identify $C(K)$ with $X_u$), $\| \vee_j y_j\| \leq C \vee_j \|y_j\|$. As in the proof of \Cref{p:krivine}, let $(x_i)$ be a partition of unity subordinate to $(\Omega_{t_i})$, and let $y_{ij} = x_i \wedge y_j$. Then $y_j \leq \sum_i y_{ij}$, hence
$$
\big\| \vee_j y_j\big\| \leq \sum_i \big\| \vee_j y_{ij}\big\| \leq \sum_i C_i \vee_j \|y_{ij}\| \leq C \vee_j \|y_j\| ,
$$
where $C = \sum_{i=1}^n C_i$.
\end{proof}

\begin{lemma}\label{l:just_right_bases}
 Suppose a separable Banach lattice $X$ is not lattice isomorphic to an AM-space, while $s(X) = \infty$.
 Then there exist disjoint vectors $e_{kn}, f_{kn} \in X_+$ ($n \in \N$, $1 \leq k \leq M_n$) so that, for every $n$,
 \begin{enumerate}
  \item $(e_{kn})_{k=1}^{M_n}$ is $2$-equivalent to the $\ell_\infty^{M_n}$ basis.
  \item $\| \sum_k f_{kn} \| > n$, while $\| \sum_k \alpha_k f_{kn} \| \leq ( \sum_k \alpha_k^2 )^{1/2}$ for any $(\alpha_k)$.
 \end{enumerate}
\end{lemma}

For the proof of this lemma, we shall use a characterization of AM-spaces from \cite[Lemma 4]{CL}.
The completion of a normed lattice $X$ is not lattice isomorphic to an AM-space iff for any $K$ we can find disjoint $x_1, \ldots, x_n \in X_+$ so that $\|\sum_k x_k\| > K$, while $\sum_k | x^*(x_k) |^2 \leq 1$ for any $x^* \in X^*$ with $\|x^*\| \leq 1$.
Equivalently,  the operator $T : X^* \to \ell_2^n : x^* \mapsto (x^*(x_k))_k$ is contractive. By duality, this is equivalent to $T^* : \ell_2^n \to X \subseteq X^{**} : e_i \mapsto x_i$ being contractive. In other words, we require that  $\| \sum_k \alpha_k x_k \| \leq 1$ whenever $\sum_k \alpha_k^2 \leq 1$.

\begin{proof}
Since the vectors in question can be produced recursively, it suffices to prove the following: Suppose $X$ is separable and is not lattice isomorphic to an AM-space. Then, for any $n$, there exists $M \in \N$ and mutually disjoint positive norm one vectors $e_1, \ldots, e_M, f_1, \ldots, f_M$, so that:
\begin{itemize}
    \item $(e_k)_{k=1}^M$ is $2$-equivalent to the $\ell_\infty^M$ basis.
    \item $\| \sum_{k=1}^M f_k \| > n$, while $\| \sum_k \alpha_k f_k \| \leq ( \sum_k \alpha_k^2 )^{1/2}$ for any $(\alpha_k)$.
    \item There exists a closed sublattice $X' \subset X$, disjoint from $e_1, \ldots, e_M, f_1, \ldots, f_M$ and not lattice isomorphic to an AM-space.
\end{itemize}

By \Cref{l:not AM}, there exists $t \in K$ so that, for any open $\Omega \ni t$, the completion of $X^\Omega$ is not lattice isomorphic to an AM-space. Find norm one positive disjoint $f_1, \ldots, f_{M+1} \in X$ so that $\| \sum_{k=1}^{M+1} f_k \| > n+1$, while $\| \sum_k \alpha_k f_k \| \leq ( \sum_k \alpha_k^2 )^{1/2}$ for any $(\alpha_k)$.
By removing one of the vectors (say $f_{M+1}$) and perturbing the rest, we may assume that $f_1, \ldots, f_M$ are supported outside of some open $\Omega \ni t$, where $t$ is a special point whose existence is guaranteed by \Cref{l:not AM}.
Let $X_0$ be the closure of $X^\Omega$. Repeat the same reasoning to construct suitable $e_1, \ldots, e_M \in X_0$, disjoint from suitable $X'$.
\end{proof}

\begin{proof}[Proof of \Cref{p:uncond_not_absolute}] 
As noted before, we can assume that $X$ is separable. If $p = s(X) < \infty$, use Krivine's Theorem to find disjoint lattice copies of $\ell_p^{2^n}$ in $X$.
Each of these contains ``independent discrete Rademachers'' $r_{kn}$ ($1 \leq k \leq n$). We know that $\| \sum_k \alpha_k r_{kn} \| \sim \big( \sum_k \alpha_k^2 \big)^{1/2}$ while $\| \sum_k |\alpha_k r_{kn}| \| = \sum_k |\alpha_k|$, hence the sequence obtained by concatenating the $r_{kn}$'s is  unconditional but not absolute. Actually, this sequence is also bibasic, due to \cite{TT19}.
 \medskip
 
 Now suppose that $s(X) = \infty$ but $X$ is not an AM-space. Use \Cref{l:just_right_bases} to find the sequences $(e_{kn})$ and $(f_{kn})$ in $X_+$ ($n \in \N$, $1 \leq k \leq 2^n$) so that the vectors involved are disjoint, $(e_{kn})$ is $2$-equivalent to the $\ell_\infty^{2^n}$ basis, $\|e_{kn}\| = 1$, and the vectors $(f_{kn})$ are such that $\| \sum_k f_{kn} \| \nearrow \infty$, while $\| \sum_k \alpha_k f_{kn} \| \leq 1$ whenever $\sum_k \alpha_k^2 \leq 1$.
 Let $H_n = (h_{ijn})_{i,j=1}^{2^n}$ be the Hadamard matrix of size $2^n \times 2^n$ (see \cite[Section 8]{TT19}).
 The entries of this matrix are equal to $\pm 1$, and the rows $\hat{h}_{kn} = (h_{kjn})_j$ are mutually orthogonal.
 \medskip
 
 Taking inspiration from \cite{APY} we let, for $1 \leq k \leq 2^n$,
 $$
 u_{kn} = e_{kn} + 2^{-n} \sum_{j=1}^{2^n} h_{kjn} f_{jn} .
 $$
 Then $(u_{kn})$ is a (double-indexed) unconditional basic sequence. Indeed, by disjointness, it suffices to establish the unconditionality of $(u_{kn})$ for a fixed $n$.
 In fact, we shall show that $(u_{kn})$ is $3$-equivalent to the $\ell_\infty^{2^n}$ basis.
 For any $(\alpha_k)$ in $c_{00}$ we have
\begin{eqnarray*}
     \vee_k |\alpha_k| &\leq& \big\| \sum_k \alpha_k e_{kn} \big\| \leq \big\| \sum_k \alpha_k u_{kn} \big\|\\
     &\leq& \big\| \sum_k \alpha_k e_{kn} \big\| + 2^{-n} \big\| \sum_k \alpha_k \sum_j h_{kjn} f_{kn} \big\| .
\end{eqnarray*}
Moreover,  $\big\| \sum_k \alpha_k e_{kn} \big\| \leq 2 \vee_k |\alpha_k|$. Therefore, it suffices to show that
 $$
 \big\| \sum_k \alpha_k \sum_j h_{kjn} f_{jn} \big\| \leq 2^n {\textrm{  whenever  }} \vee_k |\alpha_k| \leq 1 .
 $$
 By the properties of the vectors $f_{jn}$,
 $$
 \big\| \sum_k \alpha_k \sum_j h_{kjn} f_{jn} \big\|^2 = \big\| \sum_j \big( \sum_k \alpha_k h_{kjn} \big) f_{jn} \big\|^2 \leq \sum_j \big|\sum_k \alpha_k h_{kjn} \big|^2 .
 $$
 However,
 $$
 \sum_j\big| \sum_k \alpha_k h_{kjn} \big|^2 = \big\| \sum_k \alpha_k \hat{h}_{kn} \big\|^2 = 
 \sum_k |\alpha_k|^2 \|\hat{h}_{kn}\|^2 = 2^{2n} ,
 $$
 which is the desired inequality.
 \medskip
 
 Finally, we note that the sequence $(u_{kn})$ is not absolute. Indeed, for  each $n$, $\|\sum_k \pm u_{kn}\| \leq 3$, yet $\|\sum_k |u_{kn}| \| \geq \| \sum_k f_{kn} \|\nearrow \infty$.
\end{proof}

\subsection{Complemented absolute sequences}\label{complemented}
In this subsection, we prove some additional results that require conditions on the ambient space $X$, or how $[x_k]$ sits inside of $X$.
\medskip

Recall that a sequence $(x_k)$ is disjoint if and only if $\sum_{k=1}^n|a_kx_k|=\bigvee_{k=1}^n|a_kx_k|$ for all scalars $a_1,\dots,a_n$. Our next result shows that  \emph{complemented} absolute sequences behave very much like disjoint sequences.
\begin{proposition}\label{???}
Let $(x_k)$ be an absolute sequence in a Banach lattice $X$ and suppose that $[x_k]$ is complemented in $X$. Then $$\bigg\|\sum_{k=1}^n |a_kx_k|\bigg\|\sim \bigg\|\bigvee_{k=1}^n |a_kx_k|\bigg\|.$$
\end{proposition}
\begin{proof}
By the remark after \cite[Theorem 1.d.6]{LT2} we have 
$$\bigg\|\sum_{k=1}^n a_kx_k\bigg\|\sim \bigg\|\left(\sum_{k=1}^n |a_kx_k|^2\right)^{\frac{1}{2}}\bigg\|.$$
Take $\theta = \frac12$, then $\frac{1}{2}=\frac{\theta}{1}+\frac{1-\theta}{\infty}$; applying \cite[Proposition 1.d.2 (i) and (ii)]{LT2} we get that 
$$\bigg\|\sum_{k=1}^n a_kx_k\bigg\|\leq C \bigg\|\left(\sum_{k=1}^n |a_kx_k|^2\right)^{\frac{1}{2}}\bigg\| \leq C \bigg\|\sum_{k=1}^n |a_kx_k|\bigg\|^\theta \bigg\| \bigvee_{k=1}^n|a_kx_k|\bigg\|^{1-\theta},$$
where $C$ is the constant of equivalence. 
\medskip

Due to the absoluteness of $(x_k)$, there exists a constant $M^*$ so that, for any $(a_k)$, $\| \sum_{k=1}^n |a_k x_k| \| \leq M^* \| \sum_{k=1}^n a_k x_k \|$. Thus,
$$
\bigg\|\sum_{k=1}^n |a_kx_k| \bigg\|\leq  C M^* \bigg\|\sum_{k=1}^n |a_kx_k|\bigg\|^{1/2} \bigg\| \bigvee_{k=1}^n|a_kx_k|\bigg\|^{1/2} ,
$$
which leads to 
$$
\bigg\|\sum_{k=1}^n |a_kx_k| \bigg\|\leq  (C M^*)^2 \bigg\| \bigvee_{k=1}^n|a_kx_k|\bigg\| .
$$
Consequently,
$$
\bigg\|\sum_{k=1}^n a_kx_k\bigg\|\leq \bigg\|\sum_{k=1}^n |a_kx_k| \bigg\| \leq (C M^*)^2 \bigg\| \bigvee_{k=1}^n|a_kx_k|\bigg\| ,
$$
and on the other hand,
$$
\bigg\| \bigvee_{k=1}^n|a_kx_k|\bigg\| \leq \bigg\|\sum_{k=1}^n |a_kx_k| \bigg\| \leq M^* \bigg\|\sum_{k=1}^n a_kx_k\bigg\| .
$$
This completes the proof.
\end{proof}
\Cref{???} allows us to give a new Banach lattice proof of the well-known characterization of complemented unconditional sequences in $C[0,1]$.
\begin{corollary}\label{Purely banach space}
The only complemented semi-normalized unconditional basic sequences in AM-spaces are those equivalent to the unit vector basis of $c_0$. 
\end{corollary}
\begin{proof}
Suppose that $(x_k)$ is such a sequence. In AM-spaces, unconditional is the same as absolute, so $(x_k)$ is absolute. Now, using the AM-property and \Cref{???} we see that

 \begin{eqnarray*}
     \max_k{|a_k|} &\lesssim& \|\sum_{k=1}^na_kx_k\|\lesssim \|\sum_{k=1}^n|a_kx_k|\|\sim \|\bigvee_{k=1}^n |a_kx_k|\|\\
     &=&\bigvee_{k=1}^n\|a_kx_k\|\sim \max_k{|a_k|}.
 \end{eqnarray*}

\end{proof}
\begin{remark}
The proof of \Cref{Purely banach space} shows us that the complementability assumption is critical in \Cref{???}. Since $C[0,1]$ is universal, it contains  copies of every normalized unconditional basis, and the conclusion of \Cref{???} must fail for all of them except $c_0$.
\end{remark}

\begin{remark}
Of course, \Cref{Purely banach space} is well-known; it is actually known (\cite[p.~74]{DJT}) that the only complemented semi-normalized unconditional basic sequences in $\mathcal{L}_\infty$-spaces are those equivalent to the unit vector basis of $c_0$. 
\end{remark}

\subsection{Bibasic sequences in non-atomic Banach lattices}\label{Ambient}
We now consider the case when the ambient lattice is $L_p$. It was shown in \cite{TT19} that $L_1$ does not reasonably embed into the span of a bibasic sequence, so it would be interesting to know if $L_1$ admits a uniformly quasi-greedy basis. Although we do not know the answer to this question, we will prove an $L_p$-version of it. Specifically, the next proposition proves that $L_p$ cannot admit a permutable bibasis, which is in contrast to the fact that uniformly quasi-greedy bases are ``almost" permutable (in the same sense that quasi-greedy bases are ``almost" unconditional) and that $L_p$ \textit{does} admit uniformly quasi-greedy bases when $p>1$. 
\medskip

In the next proposition, we use the concept of  unbounded convergence. Given a convergence $\xrightarrow{\tau}$ on a vector lattice $X$, a net $(x_\alpha)$ is said to \emph{unbounded $\tau$-converge} to $x\in X$ if $|x_\alpha-x|\wedge u\xrightarrow{\tau}0$ for all $u\in X_+$, see \cite{T1}. When $\tau$ is the convergence of a locally solid topology,  the convergence $\xrightarrow{u\tau}$ on $X$ defines the weakest locally solid topology on $X$  agreeing with $\tau$ on the order intervals. On the other hand, unbounded order convergence acts as the natural generalization of almost everywhere convergence to vector lattices. For a comprehensive study of such convergences, see \cite{DOT17,T2, MR3803666}. Although we  do not pursue it here, we note that the unbounded convergences could provide a systematic solution to \cite[Problem 12.3]{AABW19} as they encompass convergence in measure, convergence almost everywhere, as well as various  convergences that are weaker than the norm.
\begin{proposition}\label{Lp no perm}
Suppose that $1\leq p<\infty$ and $(E_n)$ is a sequence of subspaces of a Banach lattice $X$ which forms a permutable bi-FDD of $[E_n]$. Then there is no isomorphic embedding $T:L_p\to [E_n]$ with the property that $T^{-1}:T(L_p)\subseteq X\to L_p$ maps uniformly null sequences in $T(L_p)$ to $uo$-null sequences in $L_p$.
\end{proposition}
\begin{proof}
By \cite[Corollary 9 and Remark 10, p.~102]{KS89} no Haar type system in $L_p$ ($1\leq p<\infty$) is a permutable $uo$-bibasic sequence in $L_p$. Now proceed as in \cite[Theorem 5.1]{GKP}, using the stability results proven in \cite{TT19}.
\end{proof}
We next present a result of a similar spirit for absolute FDD's.  Recall that the sequentially $u$-to-$u$-continuous isomorphisms (see \cite{TT19}) are the natural morphisms which preserve bibasic and uniformly quasi-greedy basic sequences. In \cite[Proposition 7.10]{OTTT} it is shown that a linear map is sequentially $u$-to-$u$-continuous if and only if it is multibounded, in the sense that there exists $M\geq 1$ such that for any $m$ and $x_1,\dots,x_m$ we have $\|\bigvee_{k=1}^m|Tx_k|\|\leq M\|\bigvee_{k=1}^m|x_k|\|$. Such maps are  sometimes called $(\infty,\infty)$-regular, and have been studied by many authors. 
We show that such maps take absolute sequences to absolute sequences. More precisely, we have the following proposition.

\begin{proposition}\label{abs to abs}
Suppose $X, Y$ are Banach lattices, $E$ is a subspace of $X$, and $T : E \to Y$ is multibounded, with a bounded inverse. If a sequence $(x_k) \subset E$ is absolute, then the same is true for $(Tx_k)$.
\end{proposition}

\begin{proof}
The absoluteness of $(x_k)$ means the existence of a constant $C_0$ with the property that, for every $(a_i)_{i=1}^n$, $\| \sum_i |a_i x_i| \| \leq C_0 \| \sum_i a_i x_i\|$.
$T$ being multibounded gives us the existence of a constant $C_1$ so that the inequality $\big\| \vee_{k=1}^m |T e_k| \big\| \leq C_1 \big\| \vee_{k=1}^m |e_k| \big\|$ holds. If, in addition, for any $k$ there exists an $\ell$ so that $e_\ell = - e_k$ (the sequence $(e_k)$ is ``symmetric'' -- it contains the opposite of any of its elements), then we have $\big\| \vee_{k=1}^m T e_k \big\| \leq C_1 \big\| \vee_{k=1}^m e_k \big\|$. 
For $a_1, \ldots, a_n \in \R$, $\sum_i |a_i x_i| = \bigvee_{\vp_i = \pm 1} \sum_i \vp_i a_i x_i$, and the family $\big(\sum_i \vp_i a_i x_i\big)_{\vp_i = \pm 1}$ is symmetric in the above sense, hence
\begin{align*}
\bigg\| \sum_i |a_i  T x_i| \bigg\| 
&
= \bigg\| \bigvee_{\vp_i = \pm 1} T \big( \sum_i \vp_i a_i x_i\big) \bigg\| 
\\
&
\leq C_1 \bigg\| \bigvee_{\vp_i = \pm 1}  \big( \sum_i \vp_i a_i x_i\big) \bigg\| = C_1 \bigg\| \sum_i |a_i x_i| \bigg\| .
\end{align*}
However, by the boundedness of $T^{-1}$, there exists $C_2 > 0$ so that, for any $(a_i)$, $\|\sum_i a_i x_i\| \leq C_2 \|\sum_i a_i  T x_i\|$. Then 
$$
\bigg\| \sum_i |a_i T x_i| \bigg\| \leq C_1 \bigg\| \sum_i |a_i x_i| \bigg\| \leq C_1 C_0 \bigg\| \sum_i a_i x_i \bigg\| \leq C_2 C_1 C_0 \bigg\| \sum_i a_i T x_i \bigg\| ,
$$
which shows the absoluteness of $(Tx_i)$.
\end{proof}

\begin{proposition}\label{no measurable functions}
If a $\sigma$-order complete Banach lattice embeds with multibounded inverse into the span of an absolute FDD then it is purely atomic.
\end{proposition}
\begin{proof}
Let $X$ be a $\sigma$-order complete Banach lattice, $E$ a Banach lattice, $(x_k)$ an absolute basic sequence in $E$ (which can be weakened to FDD, but we use basic for ease of reference), and $T:X\to [x_k]\subseteq E$ an isomorphic embedding with multibounded  inverse. We begin with a few reductions.
\medskip

Clearly, $X$ must be separable; let us assume that it does not have atoms. Since every separable $\sigma$-order complete nonatomic Banach lattice $X$ can be represented as a K\"othe function space on $[0,1]$ with $L_\infty \subseteq X \subseteq L_1$, we can assume that $X$ is K\"othe. 
\medskip


The point of assuming that $T$ is an embedding with multibounded inverse is that the inverse map sends absolute sequences to absolute sequences by \Cref{abs to abs}. Suppressing $T$, we view $X\subseteq [x_k]\subseteq E$. The combination of $X$ being separable and $\sigma$-order complete yields that $X$ is order continuous. By \cite[Corollary 3.1.25]{Lin}, the Rademacher's form a weakly null sequence in $X$. Hence, by passing to a subsequence and using the Bessaga-Pe{\l}\-czy{\'n}\-ski's selection principle, we may find a block sequence $(y_k)$ of $(x_k)$ such that $\|y_k-r_k\|\rightarrow 0$. Here, $(r_k)$ is a subsequence of the Rademacher's. Passing to further subsequences, we may assume that $\|y_k-r_k\|\rightarrow 0$ sufficiently fast so that $(r_k)$ is a small perturbation of $(y_k)$, and hence is absolute, by the stability results proved in \cite{TT19} (more specifically, unsuppressing $T$ we get that $(Tr_k)$ is absolute, hence $(r_k)$ is as well by the multibounded inverse assumption). Therefore, 
$$\|\one\|_X\sum_{k=1}^n |a_k| = \|\sum_{k=1}^n |a_kr_k|\|_X\leq A \|\sum_{k=1}^n a_kr_k\|_X,$$ so that a subsequence of the Rademacher's is equivalent to the unit vector basis of $\ell_1$. This means that the unit vector basis of $\ell_1$ is weakly null, a contradiction. 
\medskip


Until this point, we assumed that $X$ was atomless, and we now reduce to the case that $X$ is not purely atomic. Since $X$ is $\sigma$-order complete and separable, it is order continuous, and hence has the projection property. Let $B$ be the band generated by the atoms, so that $B\oplus B^d=X$. $B^d$ is atomless and $\sigma$-order complete, so $X$ cannot be nicely embedded into the span of an absolute FDD without  $B^d$ being as well. This concludes the proof.
%
\end{proof}
\begin{remark}
In \cite{LW} (see also \cite{Mc95}) a nonatomic AM-space $X$ is constructed that is linearly isomorphic to $c_0$, hence has an unconditional basis, which is absolute since $X$ is AM. Hence, $\sigma$-order completeness cannot be dropped in \Cref{no measurable functions}. 
\end{remark}
\begin{remark}
     The reader may check that a bi-FDD version of \Cref{Blocking to get a-FDD} is  valid, which  when combined with \Cref{Lp no perm} leads to an interesting phenomena:  Start with any unconditional FDD of $L_p$, $p>1$. Then one can block it so that it is a bi-FDD. After that, using \Cref{Lp no perm}, one can rearrange the blocked FDD so that it fails to be a bi-FDD. However, one can then find a further blocking of this blocked and rearranged FDD to regain the bi-property, and so on ad infinitum.
\end{remark}




\bibliographystyle{plain}
\bibliography{refs.bib}

\end{document}